\documentclass{amsart}

\usepackage[colorlinks=true, urlcolor=black, citecolor=black, linkcolor=black, hyperfootnotes=true]{hyperref}
\usepackage{amssymb}
\usepackage{aliascnt}%for theorem counters

\numberwithin{equation}{section}

\newcounter{dummy}
\usepackage{enumitem}
\makeatletter
\newcommand\myitem[1][]{\item[#1]\refstepcounter{dummy}\def\@currentlabel{#1}}
\makeatother

\newtheorem{thm}{Theorem}[section]

\newaliascnt{prp}{thm}
\newtheorem{prp}[prp]{Proposition}
\aliascntresetthe{prp}

\newaliascnt{cor}{thm}
\newtheorem{cor}[cor]{Corollary}
\aliascntresetthe{cor}

\newaliascnt{lm}{thm}
\newtheorem{lm}[lm]{Lemma}
\aliascntresetthe{lm}

\theoremstyle{definition}

\newaliascnt{dfn}{thm}
\newtheorem{dfn}[dfn]{Definition}
\aliascntresetthe{dfn}

\newaliascnt{xpl}{thm}
\newtheorem{xpl}[xpl]{Example}
\aliascntresetthe{xpl}

\newaliascnt{rmk}{thm}
\newtheorem{rmk}[rmk]{Remark}
\aliascntresetthe{rmk}

\newtheorem{ass}{Assumption}

\newtheorem{asss}{Assumption$'$}

\newtheorem{assss}{Assumption$''$}

\author{Tristan Bice and Lisa Orloff Clark}
\address{Institute of Mathematics of the Czech Academy of Sciences, \v{Z}itn\'a 25, 115 67 Prague, Czech Republic}
\email{tristan.bice@gmail.com}
\address{School of Mathematics and Statistics, Victoria University of Wellington, P.O. Box 600, Wellington 6140, New Zealand.}
\email{lisa.clarkf@vuw.ac.nz}
\thanks{The first author is supported by the GA\v{C}R project EXPRO 20-31529X and RVO: 67985840}
\thanks{The second author is supported by a Marsden Fund of the Royal Society of New Zealand.}

\title{Reconstructing \'Etale Groupoids from Semigroups}

\begin{document}

\begin{abstract}
We unify various \'etale groupoid reconstruction theorems such as:
\begin{itemize}
\item[\cite{Kumjian1986} \& \cite{Renault2008}] Kumjian-Renault's reconstruction from a groupoid C*-algebra;
\item[\cite{Exel2010}] Exel's reconstruction from an ample inverse semigroup;
\item[\cite{Steinberg2019}] Steinberg's reconstruction from a groupoid ring; and
\item[\cite{ChoiGardellaThiel2019}] Choi-Gardella-Thiel's reconstruction from a groupoid $L^p$-algebra.
\end{itemize}
We do this by working with certain \emph{bumpy} semigroups $S$ of functions defined on an \'etale groupoid $G$.  The semigroup structure of $S$ together with the diagonal subsemigroup $D$ then yields a natural domination relation $\prec$ on $S$.  The groupoid of $\prec$-ultrafilters is then isomorphic to the original groupoid $G$.
\end{abstract}

\maketitle

\section{Introduction}

\subsection{Motivation}

Continuous functions on \'etale groupoids provide a rich array of non-commutative structures like C*-algebras, $L^p$-algebras, Steinberg algebras and inverse semigroups.  Under suitable conditions, the groupoid can be reconstructed from the corresponding algebraic structure.  For example,  the following groupoid reconstruction result is an immediate Corollary of \cite[Proposition 4.13]{Renault2008}:
\begin{cor}
\label{cor:renault}
Let $G$ and $G'$ be effective second countable locally compact Hausdorff \'etale groupoids.
If $C^*_r(G)$ and $C^*_r(G')$ are diagonally isomorphic C*-algebras, then $G$ and $G'$ are isomorphic \'etale groupoids.
\end{cor}
Other groupoid reconstruction results  appear in
\cite{Bice2019Weyl, BrownClarkanHuef2017, BrownloweCarlsenWhattaker2017, CarlsenRout2018,CarlsenRuizSimsTomforde2017, ChoiGardellaThiel2019, Exel2010, Kumjian1986,Steinberg2019}.
While the algebraic structures involved in each paper differ, what they do share is a product defined in essentially the same way.  This hints at the possibility of unifying these reconstruction theorems and even improving some of them,  by focusing just on the product structure alone.    Indeed, in the commutative case this was essentially already done in \cite{Milgram1949} \textendash\, what we demonstrate is that \cite{Milgram1949} admits a non-commutative extension which unifies these reconstruction theorems.

\subsection{Outline}

We start with some general notational conventions in \autoref{Notation}.  We also observe that semigroup-valued (partial) functions on a groupoid admit a well-defined product given in \eqref{FunctionProduct}, so long as the domain of at least one of the functions is a bisection.  This allows us more freedom than the ring or field-valued (total) functions usually considered elsewhere, and also avoids the need to consider the infinite sums that can arise when dealing with the convolution product.

In \autoref{BumpySemigroups}, we introduce the central concept of a bumpy semigroup.  These are semigroups of functions defined on open bisections satisfying 3 basic properties needed to recover the underlying groupoid.  We then show that bumpy semigroups exist precisely on locally compact \'etale groupoids, at least when the unit space is Hausdorff, and that even stronger compact-bumpy semigroups encompass a wide range of examples, including those arising in the reconstruction theorems in the literature mentioned above.

We next introduce the domination relation in \autoref{Domination}, a transitive relation which can be defined on any abstract semigroup $S$ with a distinguished subsemigroup $D$.  We then characterise the domination relation on compact-bumpy semigroups via the compact containment relation in \eqref{precSubset}.

Next we consider ultrafilters with respect to the domination relation, using them to recover the \'etale groupoid in \autoref{GroupoidRecovery}.  This result, and its more categorical reformulation in \autoref{GroupoidIsomorphism}, is the crux of the paper, from which we derive all our reconstruction theorems.  As noted after \autoref{GroupoidIsomorphism}, Exel's reconstruction theorem follows immediately when the functions take values in the trivial $1$-element semigroup.

The next task is to isolate a bumpy semigroup from a given semigroup (or algebra) of functions defined on arbitrary open subsets.  In \autoref{MN} and \autoref{MZbumpy}, we show how this can be achieved with a purely multiplicative notion of normaliser.  This yields another reconstruction result in \autoref{GroupoidIsomorphism2}, which encompasses the original C*-algebraic reconstruction theorems of Kumjian and Renault, as well as the more recent $L^p$-algebraic reconstruction theorem of Choi-Gardella-Thiel.

In the penultimate section, we deal with more general semigroupoid-valued functions.  By utilising the abundance of $Z$-regular elements in compact-bumpy semigroups on ample groupoids, we obtain another reconstruction theorem in \autoref{GroupoidIsomorphism3}.  We finish this section by noting how this generalises Steinberg's reconstruction theorem and how this applies to groupoids where the fibres of the isotropy group bundle have no non-trivial units.

In the last section, we show how to extend all our results to the graded context.

\subsection{Acknowledgements}

The authors would like to thank Astrid an Huef for several conversations and contributions which helped improve the paper.  The first author also thanks her for her kind hospitality while he was visiting Victoria University of Wellington in January 2020

%%%%%%%%%%%%%%%%%%%%%%%%%%%%%%%%%

\section{Notation}\label{Notation}

First let us set out some general notation.

Given $a\subseteq Y\times X$, denote the image of $V\subseteq X$ and preimage of $W\subseteq Y$ by
\begin{align*}
a[V]&=\{y\in Y:\exists v\in V\ ((y,v)\in a)\}.\\
[W]a&=\{x\in X:\exists w\in W\ ((w,x)\in a)\}.
\end{align*}
Denote the domain and range of $a\subseteq Y\times X$ by
\[\mathrm{ran}(a)=a[X]\qquad\text{and}\qquad\mathrm{dom}(a)=[Y]a.\]
The (partial) functions are those $a\subseteq Y\times X$ such that $a[x]$ contains at most one element, for all $x\in X$, which is then denoted by $a(x)$ as usual, i.e. $a[x]=\{a(x)\}$.  The set of all $Y$-valued functions with domains in a family $\mathcal{X}$ is denoted by
\[{}_Y\mathcal{X}=\{a\subseteq Y\times X:\mathrm{dom}(a)\in\mathcal{X}\text{ and }\forall x\in\mathrm{dom}(a)\ (a[x]\text{ is a singleton})\}.\]
For example, $Y$ could be the complex unit disk or circle which we denote by
\begin{align*}
\mathbb{D}&=\{z\in\mathbb{C}:|z|\leq1\}.\\
\mathbb{T}&=\{z\in\mathbb{C}:|z|=1\}.
\end{align*}

Given a groupoid $G$, we denote the source and range of any $g\in G$ by
\[\mathsf{s}(g)=g^{-1}g\qquad\text{and}\qquad\mathsf{r}(g)=gg^{-1}.\]
We denote the units of $G$ by $G^0$ and the composable pairs by $G^2$, i.e.
\begin{align*}
G^0&=\{\mathsf{s}(g):g\in G\}=\{\mathsf{r}(g):g\in G\}\\
G^2&=\{(g,h)\in G\times G:\mathsf{s}(g)=\mathsf{r}(h)\}.
\end{align*}
We denote the \emph{bisections} of $G$ by
\[\mathcal{B}(G)=\{B\subseteq G:BB^{-1}\cup B^{-1}B\subseteq G^0\}.\]
Equivalently, $B\subseteq G$ is a bisection precisely when, for all $g,h\in G$,
\[\mathsf{r}(g)=\mathsf{r}(h)\quad\text{or}\quad\mathsf{s}(g)=\mathsf{s}(h)\qquad\Rightarrow\qquad g=h.\]
The collection of subsets of $G$ is denoted by $\mathcal{P}(G)=\{X:X\subseteq G\}$.

\begin{prp}
For any groupoid $G$, semigroup $Y$ and $a,b\in{}_Y\mathcal{P}(G)$ where either $\mathrm{dom}(a)$ or $\mathrm{dom}(b)$ is a bisection, we can define $ab\in{}_Y\{\mathrm{dom}(a)\mathrm{dom}(b)\}$ by
\begin{equation}\label{FunctionProduct}
ab(gh)=a(g)b(h),\quad\text{for all }(g,h)\in G^2\cap\mathrm{dom}(a)\times\mathrm{dom}(b).
\end{equation}
In particular, this yields a semigroup operation on ${}_Y{\mathcal{B}(G)}$.
\end{prp}

\begin{proof}
Note that when $B$ is a bisection and either $g,g'\in B$ or $h,h'\in B$ then
\[gh=g'h'\qquad\Rightarrow\qquad g=g'\quad\text{and}\quad h=h'\]
(e.g. if $g,g'\in B$ then $gh=g'h'$ implies $\mathsf{r}(g)=\mathsf{r}(gh)=\mathsf{r}(g'h')=\mathsf{r}(g')$ so $g=g'$ and hence $h=g^*gh=g'^*g'h'=h'$).  Thus $ab$ is well-defined by \eqref{FunctionProduct} whenever $\mathrm{dom}(a)$ or $\mathrm{dom}(b)$ is a bisection.  If both $\mathrm{dom}(a)$ and $\mathrm{dom}(b)$ are bisections then so is $\mathrm{dom}(ab)=\mathrm{dom}(a)\mathrm{dom}(b)$ and hence this yields a binary operation on ${}_Y\mathcal{B}(G)$, which is associative because $Y$ is a semigroup.
\end{proof}

\begin{rmk}\label{0rmk}
Note that ${}_Y\mathcal{B}(G)$ always has a zero/absorbing element (namely the empty function $\emptyset$) even when $Y$ itself does not.  Indeed, we are primarily interested in $Y$ satisfying \eqref{1Cancellative} below, which means $Y$ cannot have a zero if $Y$ has at least $2$ elements.  When dealing with a field like $\mathbb{C}$ or, more generally, a domain $R$, we prefer to consider $R\setminus\{0\}$-valued partial functions rather than $R$-valued total functions \textendash\, see the remark below.
\end{rmk}

\begin{rmk}\label{Y^G}
When $R$ is a ring, it is more common to consider $R$-valued functions $_R\{G\}$ defined on the entirety of $G$ under the convolution product
\[ab(f)=\sum_{f=gh}a(g)b(h),\]
at least with some conditions on the functions to ensure the sum above makes sense.  For example, when $\mathrm{supp}(a)=[R\setminus\{0\}]a$ is bisection, the sum always makes sense because it contains at most one non-zero term, in which case $ab(gh)=a(g)b(h)$ whenever $g\in\mathrm{supp}(a)$, $h\in\mathrm{supp}(b)$ and $\mathsf{s}(g)=\mathsf{r}(h)$.  We thus get a semigroup of (total) $R$-valued functions $a$ on $G$ with bisection supports, which we can identify with their restrictions $a|_{\mathrm{supp}(a)}$.  As long as $R$ is a domain, i.e. as long as $R\setminus\{0\}$ is a subsemigroup of $Y$, this restriction map yields an isomorphism with the semigroup of (partial) $(R\setminus\{0\})$-valued functions defined on bisections of $G$.
\end{rmk}

Given a topology $\mathcal{O}(G)$ on $G$, we denote the open bisections by
\[\mathcal{B}^\circ(G)=\mathcal{B}(G)\cap\mathcal{O}(G).\]
We call $G$ \emph{\'etale} if the involution $g\mapsto g^{-1}$ and product $(g,h)\mapsto gh$ are continuous on $G$ and $G^2$ respectively,  that is, $G$ is a topological groupoid,  and the source is a local homeomorphism onto an open subset of $G$.  By \cite[Theorem~5.8]{Resende2007},  if $G$ is a topological groupoid then $G$ is \'etale when the source map is an open map.   Equivalently, $G$ is \'etale if  we have a basic inverse semigroup $\mathcal{B}\subseteq\mathcal{B}^\circ(G)$, meaning $\mathcal{B}$ is both a basis for the topology and an inverse semigroup with the usual operations, i.e.
\[O,N\in\mathcal{B}\qquad\Rightarrow\qquad ON\in\mathcal{B}\quad\text{and}\quad O^{-1}\in\mathcal{B}.\]
 (see \cite[Proposition 6.6]{BiceStarling2018}, which is based on \cite{Resende2007}).

The following observation will be useful later on.

\begin{prp}\label{CO}
If $G$ is an \'etale groupoid, $O\in\mathcal{B}^\circ(G)$ and $C\subseteq G$ then
\[\mathsf{r}[C]\subseteq\mathsf{s}[O]\text{ and }C\text{ is compact}\qquad\Rightarrow\qquad OC\text{ is compact}.\]
\end{prp}

\begin{proof}
Take any net $(g_\lambda c_\lambda)\subseteq OC$.  As $C$ is compact, $(c_\lambda)$ has a subnet $(c_{f(\gamma)})$ converging to $c\in C$.  Thus $\mathsf{s}(g_{f(\gamma)})=\mathsf{r}(c_{f(\gamma)})\rightarrow\mathsf{r}(c)=\mathsf{s}(g)$, for some $g\in O$.  As $\mathsf{s}$ is a homeomorphism on $O$, $g_{f(\gamma)}\rightarrow g$ and hence $g_{f(\gamma)}c_{f(\gamma)}\rightarrow gc\in OC$.
\end{proof}

Denote the compact and compact open bisections of $G$ by $\mathcal{B}_\mathsf{c}(G)$ and $\mathcal{B}^\circ_\mathsf{c}(G)$ respectively.  We call $G$ \emph{ample} if it has a basic inverse semigroup $\mathcal{B}\subseteq\mathcal{B}^\circ_\mathsf{c}(G)$.  As $ON=O\cap N$ for any $O,N\in\mathcal{O}(G^0)$, it follows that $G^0$ is coherent and locally compact whenever $G$ is ample.  Standard results in non-Hausdorff topology (see \cite{Goubault2013}) then show that, if $G$ is ample,
\[G^0\text{ is sober and }T_1\qquad\Leftrightarrow\qquad G^0\text{ is Hausdorff}\qquad\Leftrightarrow\qquad G\text{ is locally Hausdorff}.\]
Indeed, when $G^0$ is Hausdorff, $G$ is ample if (and only if) $G$ is \'etale and has a basis of compact open bisections \textendash\, then a product of compact bisections is again compact, by \cite[Proposition 6.4]{BiceStarling2018}, and hence $\mathcal{B}_\mathsf{c}^\circ(G)$ is a basic inverse semigroup.

%%%%%%%%%%%%%%%%%%%%%%%%%%%%

\section{Bumpy Semigroups}\label{BumpySemigroups}

In \autoref{GroupoidReconstruction}, we will show how to reconstruct an \'etale groupoid $G$ from a semigroup $S$ of `bump' functions on bisections of $G$.  For our reconstruction to work, we need some general conditions on $S$ which we introduce in the present section.  We start with some basic standing assumptions.

\begin{ass}\label{Gass}\
\begin{enumerate}
\item $G$ is both a groupoid and a topological space with Hausdorff unit space $G^0$.
\item $Y$ is unital semigroup that is $1$-cancellative, i.e. for all $x,y\in Y$,
\[\label{1Cancellative}\tag{$1$-Cancellative}xy=x\qquad\Leftrightarrow\qquad y=1\qquad\Leftrightarrow\qquad yx=x.\]
\end{enumerate}
\end{ass}

We denote the invertible elements of $Y$ by
\[Y^\times=\{y\in Y:\exists y^{-1}\in Y\ (yy^{-1}=1=y^{-1}y)\}.\]
The interior of $X\subseteq G$ is denoted by $\mathrm{int}(X)=\bigcup\{V\in\mathcal{O}(G):V\subseteq X\}$.

\begin{dfn}\label{BumpyDef}
We call $S\subseteq{}_Y\mathcal{B}^\circ(G)$ \emph{bumpy} if it satisfies all of the following.
\begin{itemize}
\myitem[($1$-Proper)]\label{1Proper} $[1]a$ is compact, for all $a\in S$.
\myitem[(Urysohn)]\label{Urysohn} The empty function $\emptyset$ is in $S$ and for all $O\in\mathcal{O}(G)$ and $g \in G$, there exists  $a\in S$ such that $\mathrm{dom}(a)\subseteq O$ and $g \in\mathrm{int}([Y^\times]a)$.
\myitem[(Involutive)]\label{Involutive} Whenever $a\in S$ and $g\in\mathrm{int}([Y^\times]a)$, there exists $b\in S$ such that
\[g^{-1}\in\mathrm{int}(\{h\in G:a(h^{-1})=b(h)^{-1}\}).\]
\end{itemize}
\end{dfn}

In other words \ref{Involutive} says that, whenever $a\in S$ takes invertible values on a neighbourhood of $g$, we have $b\in S$ which realises these inverses on a neighbourhood of $g^{-1}$.

\begin{rmk}
Say $S\subseteq{}_Y\mathcal{B}^\circ(G)$ and, whenever $g\in O\in\mathcal{O}(G)$, we have $a\in S$ with
\begin{equation}\label{Urysohn'}
\tag{Urysohn$'$}\mathrm{dom}(a)\subseteq O\qquad\text{and}\qquad g\in\mathrm{int}([Y^\times]a).
\end{equation}
If $G^0$ is open, Hausdorff and contains at least two distinct points $g,h\in G^0$ then \eqref{Urysohn'} yields $a,b\in S$ with disjoint domains in $G^0$.  If $S$ is also a semigroup then $\emptyset=ab\in S$, so in this case \eqref{Urysohn'} implies \ref{Urysohn}.
\end{rmk}

We are primarily interested in the case when $G$ is both locally compact and \'etale \textendash\, otherwise there will be no bumpy semigroups on $G$.

\begin{thm}\label{Bumpy=>LCetale}
If $S\subseteq{}_Y\mathcal{B}^\circ(G)$ is a bumpy semigroup then $G$ is locally compact and \'etale.  Conversely, if $G$ is locally compact and \'etale then we have a bumpy semigroup
\[S=\bigcup_{O\in\mathcal{B}^\circ(G)}C_0(O,\mathbb{D}\setminus\{0\})\ \ \subseteq\ \ {}_{\mathbb{D}\setminus\{0\}}\mathcal{B}^\circ(G).\]
\end{thm}

\begin{proof}
First note that $G$ has a basis of open bisections, namely $\{\mathrm{dom}(a):a\in S\}$.  Indeed, if $g\in O\in\mathcal{O}(G)$ then \ref{Urysohn} yields $a\in S$ with $\mathrm{dom}(a)\subseteq O$ and $g\in\mathrm{int}([Y^\times]a)$ so, in particular, $g\in\mathrm{dom}(a)\in\mathcal{B}^\circ(G)$.  It follows that the product is an open map on $G$, as $S$ is a semigroup and $\mathrm{dom}(a)\mathrm{dom}(b)=\mathrm{dom}(ab)\in\mathcal{B}^\circ(G)$.

Next we claim that the involution $g\mapsto g^{-1}$ is an open map.  To see this, take any $g\in O\in\mathcal{O}(G)$, so \ref{Urysohn} yields $a\in S$ with $\mathrm{dom}(a)\subseteq O$ and $g\in\mathrm{int}([Y^\times]a)$.  Then \ref{Involutive} yields $b\in S$ with
\[g^{-1}\in\mathrm{int}(\{h\in G:a(h^{-1})=b(h)^{-1}\})\subseteq([Y^\times]a)^{-1}\subseteq\mathrm{dom}(a)^{-1}\subseteq O^{-1},\]
showing $O^{-1}$ is open.  Thus $\mathcal{B}^\circ(G)$ is a basic inverse semigroup, i.e. $G$ is \'etale.

To show $G$ is locally compact, it thus suffices to show that $G^0$ is locally compact.  Accordingly, take $g\in O\in\mathcal{O}(G^0)$.  As above, \ref{Urysohn} and \ref{Involutive} yield $a,b\in S$ with $\mathrm{dom}(a)\subseteq O$ and $g=gg^{-1}\in\mathrm{int}([1]ab)$.  As $C=[1]ab$ is compact and $\mathsf{r}$ is continuous, $\mathsf{r}[C]$ is also compact.  As $\mathsf{r}$ is an open map, $\mathsf{r}[\mathrm{int}(C)]$ is an open subset containing $g$, i.e. $\mathsf{r}[C]$ is a compact neighbourhood of $g$ in $G^0$.  Moreover, $\mathsf{r}[C]\subseteq\mathsf{r}[\mathrm{dom}(a)\mathrm{dom}(b)]\subseteq\mathsf{r}[\mathrm{dom}(a)]\subseteq O$, showing that $G^0$ is locally compact.

Conversely, suppose that $G$  is locally compact and \'etale.  Since $G^0$ is Hausdorff, any open bisection $O$ of $G$ is Hausdorff and thus a locally compact Hausdorff subspace.  Consequently, we can consider the set $C_0(O,\mathbb{D}\setminus\{0\})$ of continuous $\mathbb{D}\setminus\{0\}$-valued functions on $O$  which vanish at infinity.

We claim that
\[
S=\bigcup\{C_0(O,\mathbb{D}\setminus\{0\}):O\subseteq G\text{\ is an open bisection}\}
\]
is a semigroup with respect  to  the operation defined at \eqref{FunctionProduct}.
Let $a, b\in S$. Then $\mathrm{dom}(a)$ and  $\mathrm{dom}(b)$ are open bisections. Since $G$ is \'etale, multiplication is open and hence $\mathrm{dom}(ab)=\mathrm{dom}(a)\mathrm{dom}(b)$ is an open bisection. Let $(f_\lambda)=(g_\lambda h_\lambda)$ be a net in $\mathrm{dom}(ab)$ converging to $f=gh\in\mathrm{dom}(ab)$.  Then $\mathsf{r}(g_\lambda)\to \mathsf{r}(g)$. Since $\mathsf{r}$ is a homeomorphism on $\mathrm{dom}(a)$ we get that $g_\lambda\to g$. Similarly, $h_\lambda\to h$. Since $a$ and $b$ are continuous,  $a(g_\lambda)\to a(g)$ and $b(g_\lambda)\to b(g)$. Thus $ab(f_\lambda)=a(g_\lambda)b(h_\lambda)\to a(g)b(h)=ab(f)$, and  $ab$ is continuous.

Next, we check that $ab$ vanishes at infinity. For this, fix $\epsilon>0$. We need to show that $\{f\in \mathrm{dom}(ab): |ab(g)|\geq \epsilon\}$ is a compact subset of $\mathrm{dom}(ab)$. For $f=gh\in \mathrm{dom}(a)\mathrm{dom}(b)$, if
$|a(g)b(h)|\geq \epsilon$ then $|a(g)|\geq\epsilon/\|b\|_\infty$  and $|b(h)|\geq \epsilon/\|a\|_\infty$. Thus
\begin{align*}
\{gh&\in\mathrm{dom}(a)\mathrm{dom}(b): |ab(gh)|\geq\epsilon\}\\
&\subseteq\{g\in\mathrm{dom}(a): |a(g)|\geq\epsilon/\|b\|_\infty\}\{h\in\mathrm{dom}(b): |b(h)|\geq\epsilon/\|a\|_\infty\};
\end{align*}
this is a product of two compact sets in $G$ because both $a$ and $b$ vanish at infinity.
 Since $G^0$ is Hausdorff, a product of compact sets in $G$ is again compact \cite[Proposition~6.4]{BiceStarling2018}.   Now $\{f\in \mathrm{dom}(ab): |ab(g)|\geq \epsilon\}$ is a closed  subset of a compact set, and hence is compact. Thus $ab$ vanishes at infinity and hence $ab\in S$. Thus $S$ is a semigroup.

Next, we verify that $S$ is bumpy.
Let $a\in S$. Then $[1]a$ is compact because $a$ vanishes at infinity. This gives \ref{1Proper}. Trivially the empty subset of $G$ is a bisection and hence the empty function  is in $S$. So fix $O\in\mathcal{O}$ and   $g\in G$ with $g\in O$. Since $G$ is \'etale, the open bisections of $G$ are a basis and there exists an open bisection $M$ such that $g\in M\subseteq O$. Since $G$ is locally compact, there exists a compact neighbourhood $C$ with $g\in \mathrm{int}(C)\subseteq C\subseteq M$. Since the unit space is Hausdorff, so is the open bisection $M$.  Now $M$ is locally compact Hausdorff and hence completely regular, which means we have a continuous $[0,1]$-valued function $a$ which is $1$ in a neighbourhood of $g$ and zero outside of $\mathrm{int}(C)$.  In particular, $\mathrm{dom}(a)\subseteq O$ and $g\in\mathrm{int}([\mathbb{T}]a)$, giving \ref{Urysohn}. Finally, given $a\in S$ and $g\in\mathrm{int}([\mathbb{T}]a)$ we take $b=a^*$, i.e. $\mathrm{dom}(b)=\mathrm{dom}(a)^{-1}$ and $b(h)=\overline{a(h^{-1})}$, for all $h\in\mathrm{dom}(b)$.  Then
\[g^{-1}\in\mathrm{int}([\mathbb{T}]a)^{-1}=\mathrm{int}(\{h\in G:a(h^{-1})=b(h)^{-1}\})\]
so \ref{Involutive} holds and thus $S$ is a bumpy semigroup.
\end{proof}

Actually, there are various bumpy semigroups on any locally compact \'etale $G$.  For example, we could have restricted the functions above in $S$ to those taking real values or those with compact domains.  Alternatively, we could have used more general functions with values in $\mathbb{C}\setminus\{0\}$, which corresponds to \cite{Kumjian1986}, \cite{Renault2008} and \cite{ChoiGardellaThiel2019} (see the comments after \autoref{GroupoidIsomorphism2}).  The functions could even take values in $A^\times$ for some Banach algebra $A$.

We also have the following analog for ample groupoids.

\begin{thm}\label{Bumpy=>LCample}
If $S\subseteq{}_Y\mathcal{B}^\circ_\mathsf{c}(G)$ is a bumpy semigroup then $G$ is ample.  Conversely, if $G$ is ample then we have a bumpy semigroup $S={}_{\{1\}}\mathcal{B}^\circ_\mathsf{c}(G)$.
\end{thm}

\begin{proof}
If $S\subseteq{}_Y\mathcal{B}^\circ_\mathsf{c}(G)$ is a bumpy semigroup then $G$ is locally compact and \'etale, by \autoref{Bumpy=>LCetale}.  As $\{\mathrm{dom}(a):a\in S\}$ is a basis of compact open bisections, it follows that $G$ is ample.

Conversely, if $G$ is ample then $\mathcal{B}^\circ_\mathsf{c}(G)$ is a basic inverse semigroup which we can identify with the corresponding characteristic functions
\[S={}_{\{1\}}\mathcal{B}^\circ_\mathsf{c}(G)=\{\{1\}\times O:O\in\mathcal{B}^\circ_\mathsf{c}(G)\}.\]
Moreover, $S$ is a bumpy semigroup: \ref{1Proper} holds because each $O\in\mathcal{B}^\circ_\mathsf{c}(G)$ is compact; \ref{Urysohn} holds because $\mathcal{B}^\circ_\mathsf{c}(G)$ is a basis; and \ref{Involutive} holds because $O^{-1}\in\mathcal{B}^\circ_\mathsf{c}(G)$ whenever $O\in\mathcal{B}^\circ_\mathsf{c}(G)$.
\end{proof}

These bumpy semigroups on ample groupoids correspond to \cite{Exel2010}.  We could have also considered functions with values in a (discrete) group, which corresponds to \cite{Steinberg2019} (see the comments after \autoref{GroupoidIsomorphism3}).
In each of the above examples, we could actually replace $g$ in \autoref{BumpyDef} with a compact subset.  Accordingly, we strengthen bumpy semigroups as follows.

\begin{dfn}
We call bumpy $S\subseteq{}_Y\mathcal{B}^\circ(G)$ \emph{compact-bumpy} if
\begin{itemize}
\myitem[(Compact-Urysohn)]\label{CompactUrysohn} For any $C\in\mathcal{B}_\mathsf{c}(G)$ and $O\in\mathcal{B}^\circ(G)$ with $C\subseteq O$, there exists $a\in S$ such that
\[\mathrm{dom}(a)\subseteq O\qquad\text{and}\qquad C\subseteq[Y^\times]a.\]
\myitem[(Compact-Involutive)]\label{CompactInvolutive} If $a\in S$ and $C\subseteq[Y^\times]a$ is compact then there exists $b\in S$ such that, for all $g\in C$,
\[a(g)^{-1}=b(g^{-1}).\]
\end{itemize}
\end{dfn}

\begin{prp}\label{remarks compact bumpy} Assume $S\subseteq{}_Y\mathcal{B}^\circ(G)$ and $G$ is locally compact and \'etale.  Then
\begin{align}
\label{remarks compact bumpy1}\textnormal{\ref{CompactInvolutive}}\qquad&\Rightarrow\qquad\textnormal{\ref{Involutive}}.\\
\label{remarks compact bumpy2}\textnormal{\ref{CompactUrysohn}}\qquad&\Rightarrow\qquad\textnormal{\ref{Urysohn}}.
\end{align}
Moreover, if \textnormal{\ref{CompactUrysohn}} holds then it extends to arbitrary $O\in\mathcal{O}(G)$.
 \end{prp}

\begin{proof}\
\begin{itemize}
\item[\eqref{remarks compact bumpy1}] Fix $a\in S$ and $g\in\mathrm{int}([Y^\times]a)$. Then there exists compact $C$ such that $g\in\mathrm{int}( C)\subseteq C\subseteq \mathrm{int}([Y^\times]a)$. Since $C\subseteq \mathrm{dom}(a)$, it is a compact bisection. By \ref{CompactInvolutive}, there exists $b\in S$ such that for all $h\in C$ we have $a(h)^{-1}=b(h^{-1})$.  As $G$ is \'etale, the inverse is continuous and hence $g^{-1}\in\mathrm{int}(C)^{-1}=\mathrm{int}(C^{-1})\subseteq\mathrm{int}(\{h\in G: a(h^{-1})=b(h)^{-1}\})$.

\item[\eqref{remarks compact bumpy2}] We apply \ref{CompactUrysohn} to the empty subset of $G$ to get the empty function in $S$. Next, fix an open subset $O\in \mathcal{O}(G)$ with $g\in O$. Since $G$ is \'etale, there exists an open bisection $N$ such that $g\in N\subseteq O$. Since $G$ is locally compact, there exists a compact neighbourhood $C$ of $g$ such that $g\in \mathrm{int}(C)\subseteq C\subseteq N$. Notice that $C$ is a compact bisection because $N$ is a bisection. Apply \ref{CompactUrysohn} to $C\subseteq N$ to get $a\in S$ such that $\mathrm{dom}(a)\subseteq N\subseteq O$ and  $C\subseteq [Y^\times]a$ and hence $g\in \mathrm{int}(C)\subseteq\mathrm{int}( [Y^\times]a)$.
\end{itemize}

Finally, fix $C\in\mathcal{B}_\mathsf{c}(G)$ and $O\in\mathcal{O}(G)$ such that $C\subseteq O$.  Since $G^0$ is Hausdorff and the source map is a local homeomorphism, $G$ is locally Hausdorff and  the ``bi-Hausdorff bisections'' in  \cite[Proposition 6.3]{BiceStarling2018}  are just bisections. Thus \cite[Proposition 6.3]{BiceStarling2018} applies and gives an open bisection $N$  such that  $C\subseteq N$. Now we apply  \ref{CompactUrysohn} to $C\subseteq N\cap O$ to get $a\in S$ such that $\mathrm{dom}(a)\subseteq N\cap O\subseteq O$ and $C\subseteq [Y^\times]a$.
\end{proof}

%%%%%%%%%%%%%%%%%%%%%%%%%%%%%

\section{Domination}\label{Domination}

The \emph{diagonal} of any $E\subseteq{}_Y\mathcal{P}(G)$ consists of those $d\in E$ with domains in $G^0$, i.e.
\[\mathsf{D}(E)=E\cap{}_Y\mathcal{P}(G^0)=\{d\in E:\mathrm{dom}(d)\subseteq G^0\}.\]
Our goal in next section will be to recover $G$ from a bumpy semigroup $S\subseteq{}_Y\mathcal{B}^\circ(G)$ together with its diagonal subsemigroup $D=\mathsf{D}(S)$.  To do this, we will consider ultrafilters with respect to the \emph{domination} relation $\prec$.  This relation is a purely multiplicative analogue of the one in \cite[\S 1.2.2]{Bice2019Weyl}.

\begin{dfn}\label{DominationDef}
Given a semigroup $S$ and $D\subseteq S$, we define relations
\begin{align*}
a\prec_sb\qquad&\Leftrightarrow\qquad asb=a=bsa\quad\text{and}\quad as,sa\in D;\\
a\prec b\qquad&\Leftrightarrow\qquad\exists s\in S\ (a\prec_sb).
\end{align*}
\end{dfn}

As long as $D$ is a subsemigroup of $S$, $\prec$ will be transitive.  We can also switch the subscript and the right argument by suitably modifying the left argument.

\begin{prp}\label{DominationProperties}
If $D$ is a subsemigroup of $S$ then, for all $a,b',b,c',c\in S$,
\begin{align}
\tag{Transitive}\label{Transitive}a\prec_{b'}b\prec_{c'}c\qquad&\Rightarrow\qquad a\prec_{c'}c.\\
\tag{Switch}\label{Switch}a\prec_bc\qquad&\Rightarrow\qquad bab\prec_{c}b.
\end{align}
If $0\in D$ is an absorbing element of $S$ then $0\prec_{a}b$, for all $a,b\in S$.
\end{prp}

\begin{proof}\

\noindent \eqref{Transitive}  If $a\prec_{b'}b\prec_{c'}c$ then $ac'=ab'bc'\in DD\subseteq D$.
Also $c'a = c'bb'a \in DD \subseteq D$.  Then
\[ac'c=ab'bc'c=ab'b=a=bb'a=cc'bb'a=cc'a.\]

\noindent \eqref{Switch} If $a\prec_bc$ then $babc=ba\in D$,  $cbab=ab \in D$ and $babcb=bab=bcbab$.

Lastly, if $0\in D$ is an absorbing element of $S$ then, for all $a,b\in S$, we see that  $0a=a0=0\in D$ and $ba0=0=0ab$, i.e. $0\prec_ab$.
\end{proof}

Domination on semigroups of functions on $G$ can be characterised as follows.
\begin{lm}
Given a semigroup $S\subseteq{}_Y\mathcal{B}(G)$ with subsemigroup $D=\mathsf{D}(S)$,
\begin{equation}\label{eq:domprop}
a\prec_{b'}b\quad\Leftrightarrow\quad\mathrm{dom}(a)\subseteq\mathrm{dom}(b')^{-1}(G^0\cap[1]b'b)\cap(G^0\cap[1]bb')\mathrm{dom}(b')^{-1}.
\end{equation}
\end{lm}

\begin{proof}
Say $a\prec_{b'}b$ and $g\in\mathrm{dom}(a)$, so $a(g)=ab'b(g)=a(h)b'(i)b(j)$, for some $h,i,j\in G$ with $g=hij$.  As $\mathrm{dom}(a)$ is a bisection and $\mathsf{r}(g)=\mathsf{r}(h)$, we must have $h=g$ so  $a(g)=a(h)b'b(ij)=a(g)b'b(ij)$ and hence $b'b(ij)=1$, by \eqref{1Cancellative}.  As $ab'\in D$, we must also have $hi\in G^0$ so $i=h^{-1}=g^{-1}$ and $g=hij=j\in\mathrm{dom}(b)$.  Thus $g=gg^{-1}g=hij\in\mathrm{dom}(b')^{-1}(G^0\cap[1]b'b)$.  A symmetric argument yields $g\in(G^0\cap[1]bb')\mathrm{dom}(b')^{-1}$, proving the $\Rightarrow$ part of \eqref{eq:domprop}.

Conversely, if $g\in\mathrm{dom}(a)\subseteq\mathrm{dom}(b')^{-1}(G^0\cap[1]b'b)$ then we must have $g=hij$ for some $h\in\mathrm{dom}(b')^{-1}$, $i\in\mathrm{dom}(b')$ and $j\in\mathrm{dom}(b)$ with $b'b(ij)=1$ and $ij\in G^0$.  This means $i=j^{-1}$ and $h^{-1}h=ii^{-1}$ so $h=i^{-1}=j$, as $\mathrm{dom}(b')$ is a bisection, and hence $g=hij=jj^{-1}j=j$.  Thus $a(g)=a(g)b'b(g^{-1}g)=ab'b(g)$ and, in particular, $g\in\mathrm{dom}(ab'b)$.  Moreover, if $i'\in\mathrm{dom}(b')$ and $gi'$ is defined then $ii^{-1}=g^{-1}g=i'i'^{-1}$ and hence $i'=i=g^{-1}$, as $\mathrm{dom}(b')$ is a bisection.  If $j'\in\mathrm{dom}(b)$ and $hi'j'$ is defined too (and hence in $\mathrm{dom}(ab'b)$) then again $j'=j=g$, as $\mathrm{dom}(b)$ is again a bisection.  This shows that $\mathrm{dom}(ab')\subseteq\mathrm{dom}(a)\mathrm{dom}(a)^{-1}\subseteq G^0$ so $ab'\in D$, as well as $\mathrm{dom}(a)=\mathrm{dom}(ab'b)$ so $a=ab'b$.  A symmetric argument yields $b'a\in D$ and $a=bb'a$ and hence $a\prec_{b'}b$.
\end{proof}
 
In compact-bumpy semigroups, it follows that domination corresponds to compact containment.  Specifically, for any $V,W\subseteq G$, define the relation $\Subset$ by
\[V\Subset W\qquad\Leftrightarrow\qquad\exists\text{ compact }C\subseteq G\ (V\subseteq C\subseteq W).\]

\begin{prp}
If $S\subseteq{}_Y\mathcal{B}^\circ(G)$ is a compact-bumpy semigroup and $D=\mathsf{D}(S)$,
\begin{equation}\label{precSubset}
a\prec b\qquad\Leftrightarrow\qquad\mathrm{dom}(a)\Subset[Y^\times]b.
\end{equation}
\end{prp}

\begin{proof}
Say $a\prec_{b'}b$.  Note $[1]bb'$ and $[1]b'b$ are compact, by \ref{1Proper}, and hence so are $\mathrm{dom}(b')^{-1}([1]b'b)$ and $([1]bb')\mathrm{dom}(b')^{-1}$ by \autoref{CO}.  As these compact sets are contained in the Hausdorff subspace $\mathrm{dom}(b)$ (as $G^0$ is Hausdorff and $G$ is \'etale \textendash\, see \autoref{Bumpy=>LCetale}), their intersection is also compact.  This intersection contains $\mathrm{dom}(a)$, by \eqref{eq:domprop}, and is certainly contained in $[Y^\times]b$ so $\mathrm{dom}(a)\Subset[Y^\times]b$.

Conversely, if $\mathrm{dom}(a)\Subset[Y^\times]b$ then \ref{CompactInvolutive} yields $b'\in S$ such that $b(g)^{-1}=b'(g^{-1})$, for every $g\in\mathrm{dom}(a)$.  Thus
\[\mathrm{dom}(ab')=\mathrm{dom}(a)\mathrm{dom}(a)^{-1}\subseteq G^0 \quad \text {and } \quad \mathrm{dom}(b'a) = \mathrm{dom}(a)^{-1}\mathrm{dom}(a) \subseteq G^0,\]
that is $ab',b'a\in D$.   Now
\[ab'b(g)=a(g)b'b(g^{-1}g)=a(g)=bb'(gg^{-1})a(g)=bb'a(g),\]
which gives $ab'b=a=bb'a$ and hence $a\prec_{b'}b$.
\end{proof}

Note that if all elements of $Y$ are invertible, e.g. if $Y=\mathbb{F}\setminus\{0\}$ for a field $\mathbb{F}$ like $\mathbb{R}$ or $\mathbb{C}$, then \eqref{precSubset} reduces to compact containment of domains, i.e.
\[a\prec b\qquad\Leftrightarrow\qquad\mathrm{dom}(a)\Subset\mathrm{dom}(b).\]
In particular, if the domains of elements in $S$ are already compact (e.g. if $G$ is ample and $S={}_{\{1\}}\mathcal{B}^\circ_\mathsf{c}(G)$ like in \autoref{Bumpy=>LCample}) then \eqref{precSubset} becomes
\[a\prec b\qquad\Leftrightarrow\qquad\mathrm{dom}(a)\subseteq\mathrm{dom}(b).\]

%%%%%%%%%%%%%%%%%%%%%%%

\section{Groupoid Reconstruction}\label{GroupoidReconstruction}
In this section we prove our main theorem where we reconstruct a groupoid $G$ from a bumpy semigroup $S \subseteq  {}_Y\mathcal{B}^\circ(G)$ using ultrafilters.
\begin{dfn}
Given a semigroup $S$ and $D\subseteq S$, we call $F\subseteq S$ a \emph{filter} if
\[\tag{Filter}a,b\in F\qquad\Leftrightarrow\qquad\exists c\in F\ (a,b\succ c).\]
The maximal proper filters are called \emph{ultrafilters} and are denoted by $\mathcal{U}(S)$.
\end{dfn}

Equivalently, a filter is a down-directed up-set, i.e.
\begin{align}
\tag{Up-Set}F\supseteq F^\prec=\{a\in S:\exists f\in F\ (f\prec a)\}.&\\
\tag{Down-Directed}a,b\in F\qquad\Rightarrow\qquad\exists c\in F\ (a,b\succ c).&
\end{align}
If $0\in D$ is an absorbing element of $S$ then $0$ is a $\prec$-minimum (see \autoref{DominationProperties}).  This means a filter $F$ is proper if and only if $0\notin F$, so ultrafilters are maximal filters avoiding $0$.  We consider the topology on $\mathcal{U}(S)$ generated by the basis $(\mathcal{U}_a)_{a\in S}$ where
\[\mathcal{U}_a=\{U\in\mathcal{U}(S):a\in U\}.\]
On subsets $T\subseteq S$, we also have a $^*$ operation given by
\[T^*=\{s\in S:\exists t\in T\ \exists r\in S\ (t\prec_{s}r)\}.\]

\begin{rmk}
A somewhat less general version of the next result can be found in \cite[Proposition 1.2]{Bice2019Weyl}.  The key difference is that we are not using any involution on $S$ \textendash\, the involution on $G$ is recovered from the product on $S$ via this $^*$ operation.
\end{rmk}

\begin{thm}\label{GroupoidRecovery}
If $S\subseteq{}_Y\mathcal{B}^\circ(G)$ is a bumpy semigroup with diagonal $D=\mathsf{D}(S)$ then
\[g\mapsto S_g=\{a\in S:g\in\mathrm{int}([Y^\times]a)\}\]
is a homeomorphism from $G$ onto $\mathcal{U}(S)$.  Moreover, for all $(g,h)\in G^2$,
\[S_{g^{-1}}=S_g^*\qquad\text{and}\qquad S_{gh}=(S_gS_h)^\prec.\]
Thus $\mathcal{U}(S)$ is an \'etale locally compact groupoid.
\end{thm}

\begin{proof}
Fix $g\in G$.  We start by showing $S_g$ is a filter.  Take $a,b\in S_g$.  By \ref{Involutive}, we have $a',b'\in S$ with
\[g\in O= \mathrm{int}(\{h\in G:a(h)^{-1}=a'(h^{-1})\}) \cap \mathrm{int}(\{h\in G:b(h)^{-1}=b'(h^{-1})\}).\]
By \ref{Urysohn}, we have $c\in S_g$ with $\mathrm{dom}(c)\subseteq O$.  We claim that $c\prec_{a'}a$ and $c\prec_{b'}b$.  To see this, let $h \in \mathrm{dom}(c)$ and compute
\[ca'a(h) = ca'a(hh^{-1}h) = c(h)a'(h^{-1})a(h) = c(h)a(h)^{-1}a(h) = c(h).\]
Thus $ca'a = c$ and similarly $aa'c = c$.   Since $\mathrm{dom}(ca'), \mathrm{dom}(a'c) \subseteq G^0$,
we have $a'c, ca' \in D$ and hence $c\prec_{a'}a$.  That $c\prec_{b'}b$ is similar.

To see that $S_g$ is an up-set fix $a \in S_g$ and $a\prec_{b'}b$.  By \eqref{eq:domprop},
 \[g\in\mathrm{int}([Y^\times]a)\subseteq \mathrm{dom}(a)\subseteq\mathrm{dom}(b')^{-1}[1]b'b\cap[1]bb'\mathrm{dom}(b')^{-1}\subseteq[Y^\times]b.\]
As $\mathrm{dom}(a)$ is open, $g\in\mathrm{int}([Y^\times]b)$ so $b\in S_g$ and hence $S_g$ is a filter.

For maximality, say we had a filter $F$ with $S_g\subsetneqq F$.  Take $a\in S_g$ and $b\in F\setminus S_g$, so $g\in\mathrm{int}([Y^\times]a)\subseteq\mathrm{dom}(a)$ and $g\notin\mathrm{int}([Y^\times]b)$.  As $F$ is a filter, we have $c,c',d\in S$ with $F\ni d\prec_{c'}c\prec a,b$.  Taking $C=[1]cc'\mathrm{dom}(c')^{-1}$ or $C=\mathrm{dom}(c')^{-1}[1]c'c$, note
\[\mathrm{dom}(d)\subseteq C\subseteq\mathrm{dom}(c)\subseteq\mathrm{int}([Y^\times]b)\not\ni g.\]
By \ref{1Proper}, $C$ is compact and hence closed in $\mathrm{dom}(a)$ so $g\in\mathrm{dom}(a)\setminus C\in\mathcal{O}(G)$.  By \ref{Urysohn}, we have $e\in S_g$ with $\mathrm{dom}(e)\subseteq\mathrm{dom}(a)\setminus C$.  Then the only element of $S$ below both $d$ and $e$ is the empty function, as $\mathrm{dom}(d)\cap\mathrm{dom}(e)\subseteq C\cap\mathrm{dom}(e)=\emptyset$.  Thus $\emptyset\in F$, as $F$ is a filter, so $F$ is not proper and $S_g$ must be an ultrafilter.

On the other hand, say we have an ultrafilter $U\subseteq S$ and fix $a\in U$.  Whenever $u\prec_{a'}a$, we can use \eqref{eq:domprop}  to get
\[\mathrm{dom}(u) \subseteq [1]aa'\mathrm{dom}(a')^{-1} \subseteq \mathrm{dom}(a). \]
Since $\mathrm{dom}(a)$ is Hausdorff and $[1]aa'\mathrm{dom}(a')^{-1}$  is compact, it is closed in $\mathrm{dom}(a)$ and we get
 \[\overline{\mathrm{dom}(u)} \cap \mathrm{dom}(a) \subseteq [1]aa'\mathrm{dom}(a')^{-1}\]
and both are compact.
We claim that
\begin{equation}\label{eq:opencompact}\bigcap_{u\in U}\mathrm{int}([Y^\times]u)=\bigcap_{\substack{u\in U\\ u\prec a}}\overline{\mathrm{dom}(u)} \cap \mathrm{dom}(a).\end{equation}
The forward containment is clear.  To see the reverse, fix $g$ in the right-hand side and fix $u \in U$.
Using an argument similar to the one above, there exists $b,c',c \in S$ with $U \ni b \prec_{c'} c \prec u, a$.
Taking $C$ to be the compact set  $[1]cc'\mathrm{dom}(c')^{-1}$ we get
\[ \mathrm{dom}(b) \subseteq C \subseteq \mathrm{dom}(c) \subseteq \mathrm{dom}(a) \cap \mathrm{int}([Y^\times]u).\]
Since $C$ is compact in the Hausdorff space $\mathrm{dom}(a)$, $\overline{\mathrm{dom}(b)} \cap \mathrm{dom}(a) \subseteq C$.
Now
\[g\in\overline{\mathrm{dom}(b)} \cap \mathrm{dom}(a) \subseteq  \mathrm{int}([Y^\times]u),\]
proving the claim.

Furthermore, the intersection in \eqref{eq:opencompact} is nonempty as any directed intersection of non-empty compact subsets of a Hausdorff space is again non-empty.
  For any $g\in\bigcap_{u\in U}\mathrm{int}([Y^\times]u)$, we have $U\subseteq S_g$ and hence $U=S_g$, by maximality.  So $g\mapsto S_g$ is surjective.
 For injectivity, if $g \neq h$ then there exists $O \in \mathcal{B}^\circ(G)$ such that $g \in O$ and $h \notin O$ and
 \ref{Urysohn} implies $S_g \neq S_h$.  Thus we have a bijection from $G$ onto $\mathcal{U}(S)$.  As
\[S_g\in\mathcal{U}_a\qquad\Leftrightarrow\qquad a\in S_g\qquad\Leftrightarrow\qquad g\in\mathrm{int}([Y^\times]a)\]
and $(\mathrm{int}([Y^\times]a))_{a\in S}$ is a basis for $G$, by \ref{Urysohn}, $g\mapsto S_g$ is a homeomorphism.

We certainly have $S_g^*\subseteq S_{g^{-1}}$.  Conversely, take $a\in S_{g^{-1}}$.  We already showed that $S_{g^{-1}}$ is a filter, so we have $b\in S_{g^{-1}}$ with $b\prec_{a'}a$ and hence $a'ba'\prec_aa'$, by \eqref{Switch}.  As $g^{-1}\in\mathrm{int}([Y^\times]b)\subseteq\mathrm{dom}(b)\subseteq\mathrm{int}([Y^\times]a')^{-1}$,
\[g=gg^{-1}g\in\mathrm{int}([Y^\times]a')\mathrm{int}([Y^\times]b)\mathrm{int}([Y^\times]a')\subseteq\mathrm{int}([Y^\times]a'ba')\]
so $a'ba'\in S_g$ and hence $a\in S_g^*$, showing that $S_g^*=S_{g^{-1}}$.

Likewise, whenever $(g,h)\in G^2$, we immediately see that $(S_gS_h)^\prec\subseteq S_{gh}$.  Conversely, take $a\in S_{gh}$ and then further take $a',b,b',c,d\in S$ with $S_{gh}\ni c\prec_{a'}a$ and $S_{h^{-1}}\ni d\prec_{b'}b$.  We claim that $cdb'\prec_{a'}a$.  Indeed, $a'cdb'\in DD\subseteq D$ and if $cdb'a'(ijk)=c(i)db'(j)a'(k)$ is defined then $j\in G_0$, as $db'\in D$, and hence $ca'(ik)=c(i)a'(k)$ is defined so $ijk=ik\in G^0$, as $ca'\in D$, showing that $cdb'a'\in D$ too.  Also, $aa'cdb'=cdb'$ and if $cdb'a'a(ijk)=c(i)db'(j)a'a(k)$ is defined then $j\in G^0$, as $db'\in D$, so $ca'a(ik)=c(i)a'a(k)$ is defined and hence $a'a(k)=1$, as $c\prec_{a'}a$, showing that $cdb'a'a=cdb'$.  This proves the claim and hence $a\in(S_gS_h)^\prec$, as $cdb'=S_{gh}S_{h^{-1}}S_h\subseteq S_gS_h$, showing that $(S_gS_h)^\prec=S_{gh}$.

Finally, since $S$ is a bumpy semigroup, $G$ is \'etale and locally compact by \autoref{Bumpy=>LCetale}.  Thus $\mathcal{U}(S)$ is an \'etale and locally compact groupoid as well.
\end{proof}

The upshot of \autoref{GroupoidRecovery} is that we can reconstruct $G$ from any bumpy semigroup $S\subseteq{}_Y\mathcal{B}^\circ(G)$ together with its diagonal subsemigroup $D=\mathsf{D}(S)$.  This can be rephrased without reference to ultrafilters or the precise method of reconstruction.

If $E\subseteq{}_Y\mathcal{P}(G)$ and $E'\subseteq{}_{Y'}\mathcal{P}(G')$ then we call $\phi:E\mapsto E'$ \emph{diagonal-preserving} if $\phi[\mathsf{D}(E)]=\mathsf{D}(E')$.  If there exists a diagonal-preserving isomorphism from $E$ onto $E'$ then we call $E$ and $E'$ \emph{diagonally isomorphic}.

\begin{cor}\label{GroupoidIsomorphism}
Assume $S\subseteq{}_Y\mathcal{B}^\circ(G)$ and $S'\subseteq{}_{Y'}\mathcal{B}^\circ(G')$ are bumpy semigroups, where $G'$ and $Y'$ also satisfy \autoref{Gass}.  If $\phi:S\rightarrow S'$ is a diagonal-preserving isomorphism then the map $\widetilde{\phi}:G\rightarrow G'$ defined by
\[\{\widetilde{\phi}(g)\}=\bigcap_{a\in S_g}\mathrm{dom}(\phi(a))\]
is an \'etale groupoid isomorphism $(=$ homeomorphism + groupoid isomorphism$)$.
\end{cor}

\begin{proof}
Since $S$ and $S'$ are bumpy semigroups, $G$ and $G'$ are locally compact and \'etale by  \autoref{Bumpy=>LCetale}.
 Let $a,b, b'\in S$. Then $a\prec_{b'}b$ if and only if $\phi(a)\prec_{\phi(b')}\phi(b)$. It follows that  $U\mapsto\phi[U]$ is a homeomorphism from $\mathcal{U}(S)$ onto $\mathcal{U}(S')$ with $\phi[U^*]=\phi[U]^*$ and $\phi[(UV)^\prec]=(\phi[U]\phi[V])^\prec$.  Now composing with the isomorphism $g\mapsto S_g$ from $G$ onto $\mathcal{U}(S)$ given in \autoref{GroupoidRecovery}, together with the corresponding inverse $U'\mapsto\bigcap_{a'\in U'}\mathrm{dom}(a')$ from $\mathcal{U}(S')$ to $G'$, yields the required isomorphism $\widetilde{\phi}$.
\end{proof}

In particular, we can consider \autoref{GroupoidRecovery} and \autoref{GroupoidIsomorphism} when $Y=\{1\}$, as in \autoref{Bumpy=>LCample}, which corresponds to the reconstruction in \cite[Theorem 4.8]{Exel2010}.  However, our method differs somewhat \textendash\, our reconstruction of $G$ via ultrafilters in $S$ is more in line with \cite{LawsonLenz2013}, while \cite{Exel2010} instead uses ultrafilters in $D=\mathsf{D}(S)$ to recover $G^0$ first and then uses germs to reconstruct the rest of $G$.

%%%%%%%%%%%%%%%%

\section{Semigroup Reconstruction}

With the exception of \cite{Exel2010}, most groupoid reconstruction theorems start with functions that are defined (or supported \textendash\, see \autoref{Y^G}) on arbitrary open subsets, not just open bisections.  To recover the groupoid in this case, we must first obtain a subsemigroup with bisection domains.  To do this, we examine various subsets that can be defined algebraically from the diagonal.

\begin{dfn}
Given $E\subseteq A$ and any product defined on $(A\times E)\cup(E\times A)$,
\begin{align}
\tag{Normalisers}\mathsf{N}(E)&=\{a\in A:aE=Ea\}.\\
\tag{Commutant}\mathsf{C}(E)&=\{a\in A:ad=da,\text{ for all }d\in E\}.\\
\tag{Center}\mathsf{Z}(E)&=\{a\in E:ad=da,\text{ for all }d\in E\}=E\cap\mathsf{C}(E).
\end{align}
\end{dfn}

Our intermediate goal will be to show how these subsets can be characterised by the groupoid structure when $A$ is a semigroup of functions on $G$.

First let us denote the \emph{isotropy} and \emph{isosections} of $G$ by
\begin{align}
\tag{Isotropy}G^\mathrm{iso}&=\{g\in G:\mathsf{s}(g)=\mathsf{r}(g)\}.\\
\tag{Isosections}\mathcal{I}(G)&=\{I\subseteq G:II^{-1}\cup I^{-1}I\subseteq G^\mathrm{iso})\}.
\end{align}
Equivalently, $I\subseteq G$ is an isosection iff, for all $g,h\in I$,
\[\mathsf{s}(g)=\mathsf{s}(h)\qquad\Leftrightarrow\qquad\mathsf{r}(g)=\mathsf{r}(h).\]

\begin{ass}\label{Aass}
We are given a semigroup $A\subseteq{}_Y\mathcal{O}(G)$ on which the product is given by \eqref{FunctionProduct} whenever $\mathrm{dom}(a)$ or $\mathrm{dom}(b)$ is a bisection.  From $A$ we define
\[\begin{aligned}
Z&=A\cap{}_{\mathsf{Z}(Y)}\mathcal{P}(G^0)&&=\{a\in A:\mathrm{dom}(a)\subseteq G^0\text{ and }\mathrm{ran}(a)\subseteq\mathsf{Z}(Y)\}.\\
S&=A\cap{}_Y\mathcal{B}(G)&&=\{a\in A:\mathrm{dom}(a)\in\mathcal{B}(G)\}.\\
D&= A\cap{}_Y\mathcal{P}(G^0)&&=\{a\in A:\mathrm{dom}(a)\subseteq G^0\}= \mathsf{D}(S).\\
C&=A\cap{}_Y\mathcal{P}(G^\mathrm{iso})&&=\{a\in A:\mathrm{dom}(a)\subseteq G^\mathrm{iso}\}.\\
N&=A\cap{}_Y\mathcal{I}(G)&&=\{a\in A:\mathrm{dom}(a)\in\mathcal{I}(G)\}.
\end{aligned}\]
\end{ass}

\begin{rmk}
Apart from associativity, we are not making any assumptions on the product on $A\setminus S$, although the example to keep in mind would be the convolution product when $Y=R\setminus\{0\}$ for some domain $R$.  The reason we have such freedom is that we will only be examining the normalisers/commutant/centre of subsets of $S$, for which the product of pairs outside $S$ is of no relevance.
\end{rmk}

Note $Z=D$ when $Y$ is commutative.  In general, $Z$ is still the centre of $D$ when $D$ is \emph{exhaustive} in that, for all $g\in G^0$ and $y\in Y$, we have $d\in D$ with $d(g)=y$.

\begin{prp}\label{ZE}
If $D$ is exhaustive then $Z=\mathsf{Z}(D)$.
\end{prp}

\begin{proof}
Take any $e\in D$.  If $\mathrm{ran}(e)\subseteq\mathsf{Z}(Y)$ then certainly $e\in\mathsf{Z}(D)$.  Conversely, if $\mathrm{ran}(e)\nsubseteq\mathsf{Z}(Y)$ then we have $g\in G^0$ with $e(g)\notin\mathsf{Z}(Y)$, which means we have $y\in Y$ with $e(g)y\neq ye(g)$.  As $D$ is exhaustive, we have $d\in D$ with $d(g)=y$ so $e(g)d(g)\neq d(g)e(g)$ and hence $ed\neq de$, showing that $e\notin\mathsf{Z}(D)$.
\end{proof}

Recall that a collection $\mathcal{P}\subseteq\mathcal{P}(G)$ of subsets of $G$ is $T_0$ if the elements of $\mathcal{P}$ distinguish the elements of $G$, i.e.
\[\tag{$T_0$}g,h\in G\qquad\Rightarrow\qquad\exists P\in\mathcal{P}\ (|P\cap\{g,h\}|=1).\]
Next we show $C$ is the commutant of $Z$ when $\mathrm{dom}[Z]=\{\mathrm{dom}(z):z\in Z\}$ is $T_0$.

\begin{prp}
If $\mathrm{dom}[Z]$ is $T_0$ then $C=\mathsf{C}(Z)$.
\end{prp}

\begin{proof}
Take any $a\in A$.  If $\mathrm{dom}(a)\subseteq G^\mathrm{iso}$ then certainly $a\in\mathsf{C}(Z)$.  Conversely, if $\mathrm{dom}(a)\nsubseteq G^\mathrm{iso}$, we have $g\in\mathrm{dom}(a)$ with $\mathsf{s}(g)\neq\mathsf{r}(g)$.  As $Z$ is $T_0$, we have $z\in Z$ such that $\mathsf{s}(g)\in\mathrm{dom}(z)\not\ni\mathsf{r}(g)$ and hence $g\in\mathrm{dom}(az)\setminus\mathrm{dom}(za)$ or vice versa.  Either way, $az\neq za$ so $a\notin\mathsf{C}(Z)$.
\end{proof}

Next we would like to show $N$ consists of the normalisers of $Z$.  Unfortunately this is not true in general \textendash\, see \autoref{StrictNormalisers} below.  However, we can show that the normalisers of $Z$ are sandwiched between $N$ and a large subset $M\subseteq N$ consisting of those $n\in N$ that are ``$C$-$Z$-dominated'' by elements of $S$.  More precisely, we let
\[M=\{n\in N:\exists s,t\in S\ (stn=n=nts,\ tn,nt\in C\text{ and }st,ts\in Z)\}.\]

\begin{thm}\label{MN}
If $\mathrm{dom}[Z]$ is $T_0$ then $M\subseteq\mathsf{N}(Z)\subseteq N$.
\end{thm}

\begin{proof}
Take $n\in\mathsf{N}(Z)$.  To see that $n \in N$, we show $\mathrm{dom}(n)$ is an isosection.  Suppose $g,h\in\mathrm{dom}(n)$ with $s=\mathsf{s}(g)=\mathsf{s}(h)$.  For all $z\in Z$,
\begin{align*}
s\in\mathrm{dom}(z)\qquad&\Rightarrow\qquad\{g,h\}\subseteq\mathrm{dom}(nz).\\
s\notin\mathrm{dom}(z)\qquad&\Rightarrow\qquad\{g,h\}\cap\mathrm{dom}(nz)=\emptyset.
\end{align*}
By way of contradiction suppose $\mathsf{r}(g)\neq\mathsf{r}(h)$.   Then, as $\mathrm{dom}[Z]$ is $T_0$, there exists $z\in Z$ with $\mathsf{r}(g)\in\mathrm{dom}(z)\not\ni\mathsf{r}(h)$ and hence $g\in\mathrm{dom}(zn)\not\ni h$ or vice versa.  Either way, $zn\notin nZ$, by the above implications, contradicting $n\in\mathsf{N}(Z)$.  Likewise, $\mathsf{r}(g)=\mathsf{r}(h)$ implies $\mathsf{s}(g)=\mathsf{s}(h)$ so $n\in N$ and $\mathsf{N}(Z)\subseteq N$.

On the other hand, take $n\in M$, so we have $s,t\in S$ with $stn=n=nts$, $tn,nt\in C$ and $st,ts\in Z$.  For any $z\in Z$, we claim that
\[nz=sztn.\]
To see this note that, whenever $g\in\mathrm{dom}(sztn)$ or $g\in\mathrm{dom}(nz)$, we have $h\in G$ with $\mathsf{s}(g)=\mathsf{s}(h)$, $\mathsf{r}(g)=\mathsf{r}(h)$ and
\begin{align*}sztn(g)&=s(h)z(g^{-1}g)t(h^{-1})n(g)=s(h)t(h^{-1})n(g)z(g^{-1}g)\\&=n(g)z(g^{-1}g)=nz(g),\end{align*}
as $\mathrm{dom}(tn)\subseteq G^\mathrm{iso}$, $\mathrm{dom}(z)\subseteq G^0$, $\mathrm{ran}(z)\subseteq\mathsf{Z}(Y)$ and $stn=n$ which proves the claim.

Next we show $szt \in Z$ by showing $\mathrm{ran}(szt) \subseteq \mathsf{Z}(Y)$.
As $st\in Z$, we see that, for any $e\in\mathrm{dom}(szt)$, we have $g\in\mathrm{dom}(s)$ with $e=gg^{-1}$ and
\[szt(e)=s(g)z(g^{-1}g)t(g^{-1})=s(g)t(g^{-1})z(g^{-1}g)=st(gg^{-1})z(g^{-1}g)\in\mathsf{Z}(Y).\]
Thus $szt\in Z$ and hence $nZ\subseteq Zn$.  By a dual argument, $Zn\subseteq nZ$ so $n\in\mathsf{N}(Z)$, showing that $M\subseteq\mathsf{N}(Z)$.
\end{proof}

Note that we always have
\[D\subseteq C\subseteq\mathsf{C}(Z)\subseteq\mathsf{N}(Z).\]
So while $\mathsf{N}(Z)$ may be smaller than $N$, it does at least have the same diagonal.

For the next result, we need to strengthen \ref{Involutive} to require $ab,ba\in Z$, i.e.
\begin{itemize}
\myitem[($Z$-Involutive)]\label{ZInvolutive}Whenever $a\in S$ and $g\in\mathrm{int}([Y^\times]a)$, there exists $b\in S$ such that $ab,ba\in Z$ and
\[g^{-1}\in\mathrm{int}(\{h\in G:a(h^{-1})=b(h)^{-1}\}).\]
\end{itemize}
Accordingly, we say $S$ is \emph{$Z$-bumpy} if $S$ is bumpy and satisfies \ref{ZInvolutive}.  Note that all the examples of bumpy semigroups in \autoref{BumpySemigroups} are in fact $Z$-bumpy.  In the main result of this section, i.e. \autoref{GroupoidIsomorphism2}, we need to be sure that $N$ is not too big so we require $N=S$ (in particular, this holds when $G$ is effective \textendash\, see the comments after \autoref{GroupoidIsomorphism2}).

\begin{prp}\label{MZbumpy}
If $S=N$ is $Z$-bumpy then so is $\mathsf{N}(Z)$.
\end{prp}

\begin{proof}
Say $g\in O\in\mathcal{O}(G^0)$.  We claim that we have $z\in Z$ with $\mathrm{dom}(z)\subseteq O$ and $g\in\mathrm{int}([1]z)$.  To see this, note that since $S$ satisfies \ref{Urysohn} there exists $a\in S$ with $\mathrm{dom}(a)\subseteq O$ and $g\in\mathrm{int}([Y^\times]a)$.  Then \ref{ZInvolutive} yields $b\in S$ with $ab,ba\in Z$ and
\[g^{-1}\in\mathrm{int}(\{h\in G:a(h^{-1})=b(h)^{-1}\}).\]
Thus $g=gg^{-1}\in\mathrm{int}([1]ab)$ and $\mathrm{dom}(ab)\subseteq\mathrm{dom}(a)\mathrm{dom}(b)\cap G^0\subseteq\mathrm{dom}(a)\subseteq O$, as $O\subseteq G^0$.  Thus $ab$ yields the required $z\in Z$, proving the claim.  In particular, as $G^0$ is Hausdorff and hence $T_0$, this $\mathrm{dom}[Z]$ is $T_0$.  Then $M\subseteq\mathsf{N}(Z)\subseteq N=S$, by \autoref{MN}.  As $S$ satisfies \ref{1Proper}, it follows that $\mathsf{N}(Z)$ does too.

Now take any $a\in S=N$, so \ref{ZInvolutive} yields $b\in S$ with $ab,ba\in Z$ and $g^{-1}\in O=\mathrm{int}(\{h\in G:a(h^{-1})=b(h)^{-1}\})$.  Thus $gg^{-1}\in\mathrm{int}([1]ab)$ and $g^{-1}g\in\mathrm{int}([1]ba)$ so the claim yields $c,d\in Z$ with
\begin{align*}
\mathrm{dom}(c)&\subseteq\mathrm{int}([1]ab)&\text{and}&&gg^{-1}&\in\mathrm{int}([1]c).\\
\mathrm{dom}(d)&\subseteq\mathrm{int}([1]ba)&\text{and}&&g^{-1}g&\in\mathrm{int}([1]d).
\end{align*}
Note that $a$ and $b$ witness $dbc\in M$, as $ab,ba\in Z$, $adbc,dbca\in Z\subseteq E\subseteq C$ and $badbc=dbc=dbcab$.  Also,
\[g^{-1}\in(\mathrm{int}([1]d))O(\mathrm{int}([1]c))\subseteq\mathrm{int}(\{h\in G:a(h^{-1})=dbc(h)^{-1}\}),\]
showing that $M$ and hence $\mathsf{N}(Z)$ satisfies \ref{ZInvolutive}.

Now take any $g\in O\in\mathcal{O}(G)$, so \ref{Urysohn} yields $a\in S$ with
\[\mathrm{dom}(a)\subseteq O\qquad\text{and}\qquad g\in\mathrm{int}([Y^\times]a).\]
Taking $b,c,d\in S$ as above, we see that $a$ and $b$ also witness $cad\in M$, as $ab,ba\in Z$, $bcad,cadb\in Z\subseteq C$ and $abcad=cad=cadba$.  Also,
\[g\in(\mathrm{int}([1]c))O(\mathrm{int}([1]d))\subseteq\mathrm{int}([Y^\times]cad)\]
and $\mathrm{dom}(cad)\subseteq\mathrm{dom}(a)\subseteq O$, showing that $M$  and hence $\mathsf{N}(Z)$ satisfies \ref{Urysohn}.
\end{proof}

\begin{cor}\label{GroupoidIsomorphism2}
Assume $G'$ and $Y'$ satisfy \autoref{Gass} and $A'\subseteq{}_{Y'}\mathcal{O}(G')$ also satisfies \autoref{Aass}, where we likewise define $Z'$, $D'$, $S'$, $C'$ and $N'$.  If
\begin{enumerate}
\item $D$ and $D'$ are exhaustive,
\item $S=N$ is $Z$-bumpy and $S'=N'$ is $Z'$-bumpy, and
\item $A$ and $A'$ are diagonally isomorphic semigroups
\end{enumerate}
then $G$ and $G'$ are isomorphic \'etale groupoids.
\end{cor}

\begin{proof}
If $\phi:A\rightarrow A'$ is a diagonal-preserving semigroup isomorphism, so its restriction to $\mathsf{N}(Z)=\mathsf{N}(\mathsf{Z}(D))$ (see \autoref{ZE}), which is thus diagonally isomorphic to $\mathsf{N}(Z')=\mathsf{N}(\mathsf{Z}(D'))$.  By \autoref{MZbumpy}, $\mathsf{N}(Z)$ and $\mathsf{N}(Z')$ are $Z$-bumpy semigroups.  By \autoref{GroupoidIsomorphism}, it follows that $G$ and $G'$ are isomorphic \'etale groupoids.
\end{proof}

Recall that $G$ is \emph{effective} when the interior of the isotropy is the unit space, i.e. $G^0=\mathrm{int}(G^\mathrm{iso})$.  Equivalently, $G$ is effective when every open isosection is an open bisection, i.e. $\mathcal{B}^\circ(G)=\mathcal{I}^\circ(G)=\{\mathrm{int}(I):II^{-1}\cup I^{-1}I\subseteq G^\mathrm{iso}\}$, in which case we automatically have $S=N$ (although it is quite possible to have $S=N$ even when $G$ is not effective, e.g. when we start with bumpy $S$ and just take $A=S$). This is the situation considered in many reconstruction theorems.  For example, we can consider \autoref{GroupoidIsomorphism2} when $A=C^*_r(G)$ and $A'=C^*_r(G')$ are the reduced groupoid C*-algebras of effective locally compact Hausdorff $G$ and $G'$.  Following \cite[Proposition~II.4.2]{Renault1980} we view elements of the C*-algebras as functions on $G$.   As in \autoref{Y^G}, we identify each $a$ with $a|_{\mathrm{supp}(a)}$, which is then a partial function to $Y=\mathsf{Z}(Y)=\mathbb{C}\setminus\{0\}$.  Note that the convolution product agrees with the product defined in \ref{FunctionProduct} when functions are supported on bisections.  In this case, the diagonals are automatically exhaustive and hence \autoref{GroupoidIsomorphism2} implies \autoref{cor:renault} (even without second countability).  We can also take $A=F^p_\lambda(G)$ and $A'=F^p_\lambda(G')$ to be the reduced $L^p$-algebras of $G$ and $G'$ (again identifying $a$ with $a|_{\mathrm{supp}(a)}$) in which case \autoref{GroupoidIsomorphism2} yields \cite[Corollary 5.6]{ChoiGardellaThiel2019}.

Again, the precise methods of reconstruction differ \textendash\, instead of using ultrafilters to recover $G$, \cite{Renault2008} and \cite{ChoiGardellaThiel2019} first use characters to recover $G^0$ and then use germs to recover the rest of $G$.  One advantage of ultrafilters over characters/germs is that they only depend on the product rather than the full algebra structure.  Consequently, we only need a semigroup isomorphism between the C*-algebras or $L^p$-algebras rather than an isometric algebra isomorphism.

\begin{rmk}\autoref{GroupoidIsomorphism2} does not recover all of \cite[Theorem~3.3]{CarlsenRuizSimsTomforde2017}, another reconstruction theorem for reduced groupoid C*-algebras.   That result requires the interior isotropy groups to be torsion-free abelian.  We do not know how this condition relates to our requirement that $S=N$.\end{rmk}

%%%%%%%%%%%%%%%%%%%%%%%%%%

\section{Normalisers}

Here we make some further comments about normalisers.  These are not needed for the main results, however they clarify the relationship between $M$, $N$, $\mathsf{N}(Z)$ and the ``*-normalisers'' usually defined within a C*-algebra $A$ by
\[\mathsf{N}^*(D)=\{a\in A:aDa^*\cup a^*Da\subseteq D\}.\]

First, let us call $S$ \emph{compact-$Z$-bumpy} if $S$ is compact-bumpy and
\begin{itemize}
\myitem[(Compact-$Z$-Involutive)]\label{CompactZInvolutive} If $a\in S$ and $B\subseteq[Y^\times]a$ is compact,
there exists $b\in S$ with $ab,ba\in Z$ and
\[a(g)^{-1}=b(g^{-1}),\quad\text{for all }g\in B.\]
\end{itemize}
Again the examples in \autoref{BumpySemigroups} are all compact-$Z$-bumpy.
When $S$ is compact-$Z$-bumpy, elements of $M$ can be characterised as follows.

\begin{prp}
Assume $S$ is compact-$Z$-bumpy and $a\in A$.  Then $a\in M$ if and only if there exists a compact bisection $B\in\mathcal{B}_\mathsf{c}(G)$ satisfying
\begin{equation}\label{Mconditions}
\mathsf{s}[\mathrm{dom}(a)]\subseteq\mathsf{r}[B],\quad\mathsf{r}[\mathrm{dom}(a)]\subseteq\mathsf{s}[B]\quad\text{and}\quad B\mathrm{dom}(a)\cup\mathrm{dom}(a)B\subseteq G^{\mathrm{iso}}.
\end{equation}
\end{prp}

\begin{proof}
If $a\in M$ then we have $b,c\in S$ with $ab,ba\in C$, $bc,cb\in Z$ and $abc=a=cba$.  Then $B=\mathrm{dom}(c)^{-1}[1]cb\subseteq\mathrm{dom}(b)$ satisfies \eqref{Mconditions} and is compact by \ref{1Proper}.

Conversely, say $a\in A$ and $B\in\mathcal{B}_\mathsf{c}(G)$ satisfy \eqref{Mconditions}.  Note $\mathrm{dom}(a)=\mathrm{dom}(a)BB^{-1}$ is an isosection, as $\mathsf{s}[\mathrm{dom}(a)]\subseteq\mathsf{r}[B]$ and both $\mathrm{dom}(a)B$ and $B^{-1}$ are isosections, i.e. $a\in N$.  By \cite[Proposition 6.3]{BiceStarling2018}, $B$ is contained in some $O\in\mathcal{B}^\circ(G)$ and then \ref{CompactUrysohn} yields $b\in S$ with $\mathrm{dom}(b)\subseteq O$ and $B\subseteq[Y^\times]b$ and hence $ab,ba\in C$.  Then \ref{CompactZInvolutive} yields $c\in S$ with $bc,cb\in Z$ and $b(g)^{-1}=c(g^{-1})$, for all $g\in B$, so $abc=a=cba$ and hence $a\in M$.
\end{proof}

When $G$ is effective, the characterisation of $M$ above can be further simplified.  Specifically, $a\in M$ precisely when $a$ is supported on a compact bisection.

\begin{prp}
If $S$ is compact-$Z$-bumpy and $G$ is effective then
\[M=\{a\in A:\exists B\in\mathcal{B}_\mathsf{c}(G)\ (\mathrm{dom}(a)\subseteq B)\}.\]
\end{prp}

\begin{proof}
If $\mathrm{dom}(a)\subseteq B\in\mathcal{B}_\mathsf{c}(G)$ then $B^{-1}$ satisfies \eqref{Mconditions} so $a\in M$.  Conversely, if $a\in M$ then we have $B\in\mathcal{B}_\mathsf{c}(G)$ satisfying \eqref{Mconditions}.  By \cite[Proposition 6.3]{BiceStarling2018}, we can extend $B$ to some open bisection $O\in\mathcal{B}^\circ(G)$.  Then
\[\mathrm{dom}(a)B=\mathrm{dom}(a)O\in\mathrm{int}(G^\mathrm{iso})=G^0\]
and hence $\mathrm{dom}(a)=\mathrm{dom}(a)BB^{-1}\subseteq B^{-1}\in\mathcal{B}_\mathsf{c}(G)$.
\end{proof}

In particular, if $G$ is also Hausdorff then
\[M=\{a\in A:\mathrm{cl}(\mathrm{dom}(a))\in\mathcal{B}_\mathsf{c}(G)\}.\]
Here is an example of this situation where both inclusions in \autoref{MN} are strict.

\begin{xpl}\label{StrictNormalisers}
Consider the principal groupoid/equivalence relation
\[G=\{(z,-z):z\in\mathbb{Z}\}\cup\{(z,z):z\in\mathbb{Z}\}.\]
Declare $O\subseteq G$ to be open if $(0,0)\notin O$ or $\{(n,n):n\geq0\}\setminus O$ is finite.  Let
\[A=\bigcup_{O\in\mathcal{O}(G)}C_0(O,\mathbb{C}\setminus\{0\})\approx C^*_r(G).\]
Define $a\in S=N$ and $d,u\in D=C=Z\subseteq\mathsf{N}(Z)$ by
\begin{align*}
\mathrm{dom}(a)&=\{(n,-n):n>0\}&&\text{and}&a((n,-n))&=1/n.\\
\mathrm{dom}(d)&=\{(-n,-n):n>0\}&&\text{and}&d((-n,-n))&=1/n.\\
\mathrm{dom}(u)&=\{(n,n):n\geq0\}&&\text{and}&u((n,n))&=1.
\end{align*}
Note $\mathrm{dom}(d)$ (or even any subset of $G$ containing $\mathrm{dom}(d)$) is not compact so $d\notin M$.  Also $a\notin aZ$, for if we had $z\in Z$ with $a=az$ then $z((-n,-n))=1$, for all $n\in\mathbb{N}$, i.e. $[1]z\supseteq\mathrm{dom}(e)$ is not compact which means $z$ is not proper, a contradiction.  Thus $a=ua\in Za\setminus aZ$ and hence $a\notin\mathsf{N}(Z)$, showing that $M\subsetneqq\mathsf{N}(Z)\subsetneqq N$.
\end{xpl}
Further note that $\mathsf{N}^*(D)=N$ in the example above, showing that normalisers need not be the same as the *-normalisers usually considered in C*-algebras.

%%%%%%%%%%%%%%%%%%%%%

\section{Semigroupoids}\label{Semigroupoids}

As mentioned in \autoref{0rmk}, given a domain $R$, we remove $0$ to obtain a $1$-cancellative semigroup $Y=R\setminus\{0\}$.  However, if $R$ is just a ring then the product is only partially defined on $R\setminus\{0\}$.  To deal with this, in this section we consider more general \emph{semigroupoid} $Y$,
i.e. we only have a partial associative product on $Y$.  Here associativity means $x(yz)$ is defined iff $(xy)z$ is, in which case they are equal.

A \emph{unit} $1$ in a semigroupoid $Y$ is element such that, for any $y\in Y$, both $1y$ and $y1$ are defined and equal to $y$.  We denote the invertible and regular elements of $Y$ by
\begin{align}
\tag{Invertible Elements}Y^\times&=\{y\in Y:\exists y^{-1}\in Y\ (yy^{-1}=1=y^{-1}y)\}.\\
\tag{Regular Elements}Y^\mathsf{R}&=\{y\in Y:\exists y'\in Y\ (yy'y=y\text{ and }y'yy'=y')\}
\end{align}
Note that products with invertible elements are always defined and hence $Y^\times$ is a group.  Also note that $yy'y=y$ implies both $y'yy'y=y'y$ and $yy'yy'=yy'$, i.e. both $y'y$ and $yy'$ are idempotents.  Under \eqref{1Cancellative}, the only idempotent is the unit $1$, from which it follows that $Y^\mathsf{R}=Y^\times$.

We define the centre $\mathsf{Z}(Y)=\{z\in Y:\ yz=zy \text{ for all }y \in Y\}$ as before, where $yz=zy$ means `$yz$ is defined iff $zy$ is defined, in which case they are equal'.  We call unital $Y$ \emph{indecomposable} if $1$ is the only central idempotent.  Given any $Z\subseteq\mathsf{Z}(Y)$, we define the \emph{$Z$-regular} elements by
\[Y^\mathsf{R}_Z=\{y\in Y:\exists y'\in Y\ (yy',y'y\in Z,\ yy'y=y\text{ and }y'yy'=y')\}.\]
As above, $yy'$ and $y'y$ here are both central idempotents so $Y^\mathsf{R}_Z=Y^\times$ if $Y$ is indecomposable.  Let us now make this a standing assumption, replacing the previous stronger \eqref{1Cancellative} assumption.  Our previous assumptions on $G$ remain.

\begin{asss}\label{IndecomposableSemigroupoid}\
\begin{enumerate}
\item $G$ is both a groupoid and a topological space with Hausdorff unit space $G^0$.
\item $Y$ is an indecomposable semigroupoid.
\end{enumerate}
\end{asss}

Like before in \eqref{FunctionProduct}, for any $a,b\in{}_Y\mathcal{P}(G)$ such that either $\mathrm{dom}(a)$ or $\mathrm{dom}(b)$ is a bisection, we can define $ab\in{}_Y\mathcal{P}(G)$ by
\begin{equation}\label{FunctionProduct2}
ab(gh)=a(g)b(h)
\end{equation}
whenever possible, i.e. whenever $g\in\mathrm{dom}(a)$, $h\in\mathrm{dom}(b)$, $gh$ is defined and $a(g)b(h)$ is defined (the only thing to note is that $\mathrm{dom}(ab)$ can now be a proper subset of $\mathrm{dom}(a)\mathrm{dom}(b)$).  So we can still take \autoref{Aass} to be in force, just with more general semigroupoid $Y$.

\begin{asss}\label{AassAgain}
As before in \autoref{Aass}, we assume we are given a semigroup $A\subseteq{}_Y\mathcal{O}(G)$ on which the product is given by \eqref{FunctionProduct2} whenever $\mathrm{dom}(a)$ or $\mathrm{dom}(b)$ is a bisection and from which we define subsets $Z$, $D$, $S$, $C$ and $N$, as well as
\[R=A\cap{}_{Y^\times}\mathcal{B}_\mathsf{c}(G)=\{a\in S:\mathrm{dom}(a)\text{ is compact and }\mathrm{ran}(a)\subseteq Y^\times\}.\]
\end{asss}
Our aim in what follows is to find conditions to ensure that a diagonal-preserving semigroup isomorphism from
 $A\subseteq{}_Y\mathcal{O}(G)$  to  $A'\subseteq{}_Y\mathcal{O}(G')$  restricts to a diagonal-preserving semigroup isomorphism between bumpy semigroups $R$ and $R'$ when $S$ and $S'$ are compact-bumpy.    Then we apply \autoref{GroupoidIsomorphism}  to prove \autoref{GroupoidIsomorphism3}.
\begin{prp}\label{RZS}
If $S$ is compact-bumpy then \[R=S^\mathsf{R}_Z := \{a\in S:\exists a'\in S\ (aa',a'a\in Z,\ aa'a=a\text{ and }a'aa'=a')\}.\]
\end{prp}

\begin{proof}
If $a\in S^\mathsf{R}_Z$ then we have $a'\in S$ such that, for all $g\in\mathrm{dom}(a)$,
\[a(g)a'(g^{-1})a(g)=a(g),\quad a'(g^{-1})a(g)a'(g^{-1})=a'(g^{-1})\]
and $a(g)a'(g^{-1}),a'(g^{-1})a(g)\in\mathsf{Z}(Y)$ are idempotents.  Thus $a(g)\in Y^{\mathsf{R}}_Z$ and hence $a(g)^{-1}=a'(g^{-1})$, by \autoref{IndecomposableSemigroupoid},  and $\mathrm{ran}(a)\subseteq Y^\times$.  Also $\mathrm{dom}(a)=\mathrm{dom}(a)[1]a'a$ is compact, by \ref{1Proper}, i.e. $a\in R$, showing that $S^\mathsf{R}_Z\subseteq R$.

Conversely, if $a\in R$ then  $\mathrm{dom}(a)$ is compact and $\mathrm{ran}(a)\subseteq Y^\times$ so \ref{CompactInvolutive} yields $a'\in S$ such that $a(g)^{-1}=a'(g^{-1})$, for all $g\in\mathrm{dom}(a)$.  Thus $aa'$ and $a'a$ are the characteristic functions of $\mathsf{r}[\mathrm{dom}(a)]$ and $\mathsf{s}[\mathrm{dom}(a)]$ respectively so $aa',a'a\in Z$, $aa'a=a$ and $a'aa'=a'$, showing that $a\in S^\mathsf{R}_Z$.
\end{proof}

The following lemma is a useful fact about \'etale groupoids.
\begin{lm}\label{Splitting}
If $G$ is \'etale and $O,O'\in\mathcal{B}^\circ_\mathsf{c}(G)$
then we have $O''\in\mathcal{B}^\circ_\mathsf{c}(G)$ with
\[O''\subseteq O'\quad\text{and}\quad\mathsf{s}[O'']=\mathsf{s}[O']\setminus\mathsf{s}[O]\quad(\text{or, equivalently, }\mathsf{r}[O'']=\mathsf{r}[O']\setminus\mathsf{r}[O]).\]
\end{lm}

\begin{proof}
If $G$ is \'etale and $O,O'\in\mathcal{B}^\circ_\mathsf{c}(G)$ then  $\mathsf{s}[O'],\mathsf{s}[O]\in\mathcal{O}_\mathsf{c}(G^0)$.  As $G^0$ is Hausdorff, $\mathsf{s}[O']\setminus\mathsf{s}[O]\in\mathcal{O}_\mathsf{c}(G^0)$ too so we can take $O''=O'(\mathsf{s}[O']\setminus\mathsf{s}[O])\in\mathcal{B}^\circ_\mathsf{c}(G)$.
\end{proof}

\begin{prp}
Let $G$ be an ample groupoid.  If $S$ is compact-bumpy and
\begin{equation}\label{DomainProduct}
\mathrm{dom}(ab)\subseteq\mathrm{dom}(a)\mathrm{dom}(b),
\end{equation}
for all $a,b\in N$, then
\begin{equation}\label{RZC}
C^\mathsf{R}_Z\subseteq S\qquad\Rightarrow\qquad N^\mathsf{R}_Z\subseteq S.
\end{equation}
\end{prp}

\begin{proof}
Say $C^\mathsf{R}_Z\subseteq S$ and take any $a\in N^\mathsf{R}_Z$, so we have $a'\in N$ with $aa'a=a$, $a'aa'=a'$ and $aa',a'a\in Z$.  Thus $\mathsf{s}[\mathrm{dom}(a)]\subseteq[1]a'a$, as $a=aa'a$ and $Y$ is indecomposable, and conversely $[1]a'a\subseteq\mathrm{dom}(a'a)\subseteq\mathsf{s}[\mathrm{dom}(a)]$, by \eqref{DomainProduct}.  Thus $\mathsf{s}[\mathrm{dom}(a)]=[1]a'a$ is compact, by \ref{1Proper}.  As $G$ is ample, each $g\in\mathrm{dom}(a)$ is contained in some $O\in\mathcal{B}^\circ_\mathsf{c}(G)$, which is itself contained in $\mathrm{dom}(a)$.  As $\mathsf{s}[\mathrm{dom}(a)]$ is compact, we can have a finite family of such $O$ whose sources cover $\mathsf{s}[\mathrm{dom}(a)]$.  By \autoref{Splitting}, we can shrink these if necessary to ensure the sources of these bisections are disjoint but still cover $\mathsf{s}[\mathrm{dom}(a)]$.  Taking the union of these we thus obtain $O\in\mathcal{B}^\circ_\mathsf{c}(G)$ such that
\[\mathrm{dom}(a)O^{-1}\cup O^{-1}\mathrm{dom}(a)\subseteq G^\mathrm{iso}\quad\text{and}\quad\mathrm{dom}(a)O^{-1}O=\mathrm{dom}(a)=OO^{-1}\mathrm{dom}(a).\]

Now \ref{CompactUrysohn} yields $b\in S$ with $\mathrm{dom}(b)=O$ and $\mathrm{ran}(b)\subseteq Y^\times$.  Then \ref{CompactInvolutive} yields $b'\in S$ with $b(g)^{-1}=b'(g^{-1})$, for all $g\in O$.  It follows that $ab',ba'\in C$, $ab'ba'=aa'\in Z$, $a'a=ba'ab'\in Z$, $ab'ba'ab'=aa'ab'=ab'$ and $ba'ab'ba'=ba'aa'=ba'$, i.e. $ab'\in C^\mathsf{R}_Z\subseteq S$ so $a=ab'b\in S$, showing $N^\mathsf{R}_Z\subseteq S$.
\end{proof}

\begin{cor}
Let $G$ be an ample groupoid.  If $S$ is compact-bumpy, $D$ is exhaustive and \eqref{DomainProduct} holds, for all $a,b\in N$, then
\begin{equation}\label{R}
C^\mathsf{R}_Z\subseteq S\qquad\Rightarrow\quad R=\mathsf{N}(\mathsf{Z}(D)^\mathsf{R})^\mathsf{R}_{\mathsf{Z}(D)}.
\end{equation}
\end{cor}

\begin{proof}
As $D$ is exhaustive, we immediately see that $Z=\mathsf{Z}(D)$, just like in \autoref{ZE}.  As in \autoref{RZS}, we see that
\[\mathsf{Z}(D)^\mathsf{R}=A\cap{}_{\mathsf{Z}(Y)^\times}\mathcal{O}_\mathsf{c}(G^0)=\{z\in Z:\mathrm{dom}(z)\text{ is compact and }\mathrm{ran}(z)\subseteq\mathsf{Z}(Y)^\times\}.\]
As in \autoref{MN}, $\mathsf{N}(\mathsf{Z}(D)^\mathsf{R})\subseteq N$ so $\mathsf{N}(\mathsf{Z}(D)^\mathsf{R})^\mathsf{R}_{\mathsf{Z}(D)}\subseteq N^\mathsf{R}_Z$.  By \eqref{RZC}, $N^\mathsf{R}_Z\subseteq S$ and hence $N^\mathsf{R}_Z\subseteq S^\mathsf{R}_Z=R$, by \autoref{RZS}, showing $\mathsf{N}(\mathsf{Z}(D)^\mathsf{R})^\mathsf{R}_{\mathsf{Z}(D)}\subseteq R$.  On the other hand, if $a\in R$ then $B=\mathrm{dom}(a)^{-1}$ satisfies \eqref{Mconditions} so $a\in\mathsf{N}(\mathsf{Z}(D)^\mathsf{R})$, as in \autoref{MN}, and hence $a\in\mathsf{N}(\mathsf{Z}(D)^\mathsf{R})^\mathsf{R}_{\mathsf{Z}(D)}$, by \ref{CompactInvolutive}.
\end{proof}

\begin{cor}\label{GroupoidIsomorphism3}
Again assume $A'\subseteq{}_{Y'}\mathcal{O}(G')$ satisfies the same assumptions as $A$ in \autoref{AassAgain}, from which we likewise define $Z'$, $D'$, $S'$, $C'$, $N'$ and $R'$.  If
\begin{enumerate}
\item $\mathrm{dom}(ab)\subseteq\mathrm{dom}(a)\mathrm{dom}(b)$, whenever $a,b\in N$ or $a,b\in N'$,
\item $D$ and $D'$ are exhaustive,
\item $S$ and $S'$ are compact-bumpy,
\item $G$ and $G'$ are ample,
\item $C^\mathsf{R}_Z\subseteq S$ and $C'^{\mathsf{R}}_{Z'}\subseteq S'$, and
\item $A$ and $A'$ are diagonally isomorphic
\end{enumerate}
then $G$ and $G'$ are isomorphic \'etale groupoids.
\end{cor}

\begin{proof}
If $\phi:A\rightarrow A'$ is a diagonal-preserving semigroup isomorphism then so its restriction to $R=\mathsf{N}(\mathsf{Z}(D)^\mathsf{R})^\mathsf{R}_{\mathsf{Z}(D)}$ (see \eqref{R}), which is thus diagonally isomorphic to $R'=\mathsf{N}(\mathsf{Z}(D')^\mathsf{R})^\mathsf{R}_{\mathsf{Z}(D')}$.  As $G$ is ample and $S$ is compact-bumpy, so is $R$.  Likewise, $R'$ is compact-bumpy so $G$ and $G'$ are isomorphic \'etale groupoids, by \autoref{GroupoidIsomorphism} (noting that the elements of $R$ and $R'$ all have ranges contained in $Y^\times$ and $Y'^\times$ respectively, which are bone fide semigroups, not just a semigroupoids).
\end{proof}

We can apply \autoref{GroupoidIsomorphism3} to subsemigroups  $A$ and $A'$ of Steinberg algebras which  consist of locally constant compactly supported functions on ample groupoids $G$ and $G'$ taking values in unital indecomposable rings, similar to the set-up in \cite{Steinberg2019}.  If the groupoids are Hausdorff, then \autoref{GroupoidIsomorphism3} applies when $A$ and $A'$ are the Steinberg algebras themselves.
 \autoref{GroupoidIsomorphism3} generalises \cite[Theorem 5.6]{Steinberg2019} in several respects, namely
\begin{enumerate}
\item we only require a semigroup isomorphism, not a ring isomorphism;
\item the groupoids need not be Hausdorff (as long as their unit spaces are);
\item the coefficient rings can be non-commutative (and non-identical); and
\item Steinberg's ``local bisection hypothesis'' on the algebra is reduced to a regularity condition on the isotropy subalgebra (specifically $C^\mathsf{R}_Z\subseteq S$).
\end{enumerate}
Regarding the last point, Steinberg does obtain a similar reduction in \cite[Proposition 4.10]{Steinberg2019}.  Indeed, $C^\mathsf{R}_Z$ consists of those elements of $C$ which normalise the diagonal in Steinberg's sense because the elements of $C$ and $Z$ commute and hence, for $c,c'\in C$, $cc'\in Z$ is equivalent to $(cc'Z=)cZc'\subseteq Z$.  In \cite[Theorem 5.5]{Steinberg2019}, Steinberg also shows this only needs to apply to one of the algebras, although the proof makes crucial use of the additive structure.  If a purely multiplicative analog of \cite[Theorem 5.5]{Steinberg2019} could be proved then we would likewise only need either $C^\mathsf{R}_Z\subseteq S$ or $C'^{\mathsf{R}}_{Z'}\subseteq S'$ in \autoref{GroupoidIsomorphism3}.  We will also extend \autoref{GroupoidIsomorphism3} to the graded context in \autoref{GradedGroupoidIsomorphism3} below, just like in \cite[Theorem 5.6]{Steinberg2019}.

Steinberg also helps us see that our hypothesis  $C^\mathsf{R}_Z\subseteq S$
includes some checkable groupoid conditions as follows.
For  $x\in G^0$, consider the ``interior isotropy group'' $H_x$ given by
\begin{equation}
\label{eq:intiso}
H_x=\{g\in\mathrm{int}(G^\mathrm{iso}):\mathsf{s}(g)=x=\mathsf{r}(g)\}.
\end{equation}
Define
\[C_x=\{c|_{H_x}:c\in C\} \quad \text{ and }\quad S_x=\{a|_{G_x}:a\in S\}.\]
As noted in \cite[\S2.1]{Steinberg2019},
when $C_x$ is the convolution algebra of finitely supported functions on $H_x$, then
\begin{align}
&G \text{ is effective} \notag\\
\Rightarrow\quad&\text{for each } x \in G^{0}, H_x\text{ is free, or torsion-free and abelian}\notag\\
\Rightarrow\quad&\text{for each } x \in G^{0}, H_x\text{ is left or right orderable}\notag\\
\Rightarrow\quad&\text{for each } x \in G^{0}, H_x \text{ has the unique product property}\notag\\
\Rightarrow\quad&\text{for each } x \in G^{0}, C_x  \text{ has no non-trivial units}\notag\\
\Rightarrow\quad&\text{for each } x \in G^{0}, C_x^\times \subseteq S_x\notag\\
\Rightarrow\quad&  C^\mathsf{R}_Z\subseteq S. \label{eq:crzs}
\end{align}
To summarize, for example, the following is a consequence of \autoref{GroupoidIsomorphism3}

\begin{cor}
\label{cor:steinberg}
Let $G$ and $G'$ be ample Hausdorff groupoids and let $K$ and $K'$ be indecomposable unital rings.
Suppose that   $G$ and $G'$ each satisfy any one of the conditions that imply \eqref{eq:crzs} above.
If the Steinberg algebras  $A_K(G)$ and $A_{K'}(G')$ are diagonally isomorphic as semigroups,
then $G$ and $G'$ are isomorphic groupoids.
\end{cor}
\begin{proof}  Conditions (1)--(3) of \autoref{GroupoidIsomorphism3} are automatic.  Conditions (4)--(6) are assumed so the result follows. \end{proof}

When dealing with Steinberg algebras, it even suffices to have $C_x^\times\subseteq S_x$ on some dense subset of $G^0$ (see \cite[Theorem 4.13]{Steinberg2019}) so, as with \cite[Theorem 5.6]{Steinberg2019}, \autoref{GroupoidIsomorphism3} also generalises \cite[Theorem 3.1]{CarlsenRout2018}, \cite[Theorem~6.2]{BrownClarkanHuef2017} and \cite[Theorem~3.1]{AraBosaHazratSims2017}.

%%%%%%%%%%%%%%%%%%%%%%%%%%%%%%%%%%%%5

\section{Gradings}

In this section we generalise our results to graded groupoids and semigroups.

\begin{dfn}
We call $\mathsf{c}:\Sigma\rightarrow\Gamma$ a \emph{grading} if $\Sigma$ and $\Gamma$ are semigroupoids and
\[\mathsf{c}(xy)=\mathsf{c}(x)\mathsf{c}(y),\quad\text{whenever $xy$ is defined}.\]
We call $H\subseteq\Sigma$ \emph{homogeneous} if $\mathsf{c}(x)=\mathsf{c}(y)$, for all $x,y\in H$.
\end{dfn}
In other words, a grading is a cocycle/homomorphism/functor.  Note that any semigroupoid $\Sigma$ is trivially graded by the one-element group $\Gamma=\{1\}$.  In this case, the following assumption reduces to \autoref{Gass}.

\begin{assss}\label{assss}\
\begin{enumerate}
\item $G$ is both a groupoid and a topological space with Hausdorff unit space $G^0$ and grading $\mathsf{c}:G\mapsto\Gamma$, where $\Gamma$ is another groupoid.
\item $Y$ is a semigroup satisfying \eqref{1Cancellative}.
\end{enumerate}
\end{assss}

Let $\mathcal{O}_\mathsf{h}(G)$ and $\mathcal{B}^\circ_\mathsf{h}(G)$ denote the homogeneous open subsets and bisections.  We define bumpy semigroups $S\subseteq{}_Y\mathcal{B}^\circ_\mathsf{h}(G)$ exactly like before in \autoref{BumpyDef}.  So the only extra requirement is that $\mathrm{dom}(a)$ is homogeneous, for all $a\in S$.  It follows that $\mathsf{c}$ yields a grading on $S\setminus\{\emptyset\}$, which we also denote by $\mathsf{c}$, i.e. for all $a\in S\setminus\{\emptyset\}$,
\[\mathsf{c}[\mathrm{dom}(a)]=\{\mathsf{c}(a)\}.\]

Recall from \autoref{Bumpy=>LCetale} that the existence of a bumpy semigroup implies that $G$ is locally compact and \'etale.  It now also follows that the grading is locally constant, i.e. continuous with respect to the discrete topology on $\Gamma$.

\begin{prp}\label{Bumpy=>gradedLCetale}
If $S\subseteq{}_Y\mathcal{B}^\circ_\mathsf{h}(G)$ is bumpy then $\mathsf{c}$ is locally constant.
\end{prp}

\begin{proof}
Whenever $g\in O\in\mathcal{O}(G)$, \ref{Urysohn} yields $a\in S$ defined on a neighbourhood of $g$.  As $\mathrm{dom}(a)$ is homogeneous, $\mathsf{c}[\mathrm{dom}(a)]=\{\mathsf{c}(g)\}$, showing that $\mathsf{c}$ is constant on a neighbourhood of $g$.
\end{proof}

Recall from \autoref{GroupoidRecovery} that we can recover both the topology and groupoid structure of $G$ from any bumpy semigroup.  Now the grading on $S\setminus\{\emptyset\}$ can also be used to recover the grading on $G$.

\begin{thm}\label{GradedGroupoidRecovery}
If $S\subseteq{}_Y\mathcal{B}^\circ_\mathsf{h}(G)$ is a bumpy semigroup with diagonal $D=\mathsf{D}(S)$ then
\[g\mapsto S_g=\{a\in S:g\in\mathrm{int}([Y^\times]a)\}\]
is a homeomorphism from $G$ onto $\mathcal{U}(S)$.  Moreover, for all $(g,h)\in G^2$,
\[S_{g^{-1}}=S_g^*,\qquad S_{gh}=(S_gS_h)^\prec\qquad\text{and}\qquad\mathsf{c}[S_g]=\{\mathsf{c}(g)\}.\]
\end{thm}

\begin{proof}
The proof is exactly like that of \autoref{GroupoidRecovery}.  The only extra thing to note is $\mathsf{c}[S_g]=\{\mathsf{c}(g)\}$, which is immediate from the fact that $g\in\mathrm{dom}(a)$ implies $\mathsf{c}(a)=\mathsf{c}(g)$, by the definition on $\mathsf{c}$ on $S\setminus\{\emptyset\}$.
\end{proof}

As in \autoref{GroupoidIsomorphism}, the categorical import of \autoref{GradedGroupoidRecovery} can be rephrased without reference to ultrafilters or the precise method of reconstruction.
If $G'$, $Y'$ and $\mathsf{c}'$ also satisfy \autoref{assss} (with the same $\Gamma$) then $\mathsf{c}'$ also yields a grading on any $S'\subseteq{}_{Y'}\mathcal{B}^\circ(G')$ or, more precisely, on $S'\setminus\{\emptyset\}$.  We call an isomorphism $\phi:S\rightarrow S'$ \emph{graded} if $\mathsf{c}(a)=\mathsf{c}'(\phi(a))$, for all $a\in S\setminus\{\emptyset\}$.

\begin{cor}\label{GradedGroupoidIsomorphism}
If $S\subseteq{}_Y\mathcal{B}^\circ_\mathsf{h}(G)$ and $S'\subseteq{}_{Y'}\mathcal{B}^\circ_\mathsf{h}(G')$ are diagonally isomorphic graded bumpy semigroups then $G$ and $G'$ are isomorphic graded \'etale groupoids.
\end{cor}

\begin{proof}
If $\phi:S\mapsto S'$ is a diagonal-preserving graded semigroup isomorphism then $U\mapsto\phi[U]$ is a homeomorphism from $\mathcal{U}(S)$ onto $\mathcal{U}(S')$ with $\phi[U^*]=\phi[U]^*$, $\phi[(UV)^\prec]=(\phi[U]\phi[V])^\prec$ and $\mathsf{c}'[\phi[U]]=\mathsf{c}[U]$.  Thus $G$ and $G'$ are isomorphic graded \'etale groupoids, by \autoref{GradedGroupoidRecovery}.
\end{proof}

\begin{cor}\label{GradedGroupoidIsomorphism2}
Assume $G'$, $Y'$ and $\mathsf{c}'$ also satisfy \autoref{assss}.  Further assume we are given a semigroups $A\subseteq{}_Y\mathcal{O}_\mathsf{h}(G)$ and $A'\subseteq{}_{Y'}\mathcal{O}_\mathsf{h}(G')$ on which the product is given by \eqref{FunctionProduct} whenever $\mathrm{dom}(a)$ or $\mathrm{dom}(b)$ is a bisection.  We define $Z,D,S,C,N,M\subseteq A$ and $Z',D',S',C',N',M'\subseteq A'$ as in \autoref{Aass}.  If
\begin{enumerate}
\item $D$ and $D'$ are exhaustive,
\item $S=N$ is $Z$-bumpy and $S'=N'$ is $Z'$-bumpy, and
\item $A$ and $A'$ are diagonally isomorphic graded semigroups
\end{enumerate}
then $G$ and $G'$ are isomorphic graded \'etale groupoids.
\end{cor}

\begin{proof}
If $\phi:A\rightarrow A'$ is a diagonal-preserving graded semigroup isomorphism, so its restriction to $\mathsf{N}(Z)=\mathsf{N}(\mathsf{Z}(D))$ (see \autoref{ZE}), which is thus diagonally isomorphic to $\mathsf{N}(Z')=\mathsf{N}(\mathsf{Z}(D'))$.  By \autoref{MZbumpy}, $\mathsf{N}(Z)$ and $\mathsf{N}(Z')$ are $Z$-bumpy.  By \autoref{GradedGroupoidIsomorphism}, $G$ and $G'$ are isomorphic graded \'etale groupoids.
\end{proof}

Note if $O\in \mathcal{I}^\circ_\mathsf{h}(G)$ is a homogeneous open isosection then, for some unit $\gamma\in\Gamma^0$,
\[OO^{-1}\cup O^{-1}O\subseteq\mathrm{int}(G^\mathrm{iso}\cap[\gamma]\mathsf{c}).\]
So if $[\gamma]\mathsf{c}$ is effective, for all $\gamma\in\Gamma^0$, then $\mathcal{I}^\circ_\mathsf{h}(G)=\mathcal{B}^\circ_\mathsf{h}(G)$ and hence $S=N$.  This can be useful when we have a diagonal-preserving isomorphism $\phi:A\rightarrow A'$ even if $\phi$ does not respect the grading.  Indeed, as long as $[\gamma]\mathsf{c}$ and $[\gamma]\mathsf{c}'$ are effective, for all $\gamma\in\Gamma^0$, we can then still apply the ungraded reconstruction result in \autoref{GroupoidIsomorphism2} to show that $G$ and $G'$ are isomorphic (ungraded) \'etale groupoids.

For our final graded result, we weaken our assumption on $Y$ and $Y'$ as in \autoref{Semigroupoids}.

\begin{cor}\label{GradedGroupoidIsomorphism3}
Assume $G'$, $Y'$ and $\mathsf{c}'$ also satisfy \autoref{assss}, except that $Y$ and $Y'$ are only assumed to be indecomposable semigroupoids.  Further assume we are given a semigroups $A\subseteq{}_Y\mathcal{O}_\mathsf{h}(G)$ and $A'\subseteq{}_{Y'}\mathcal{O}_\mathsf{h}(G')$ on which the product is given by \eqref{FunctionProduct2} whenever $\mathrm{dom}(a)$ or $\mathrm{dom}(b)$ is a bisection.  We define subsets $Z,D,S,C,N\subseteq A$ and $Z',D',S',C',N'\subseteq A'$ as in \autoref{Aass}.  If
\begin{enumerate}
\item $\mathrm{dom}(ab)\subseteq\mathrm{dom}(a)\mathrm{dom}(b)$, whenever $a,b\in N$ or $a,b\in N'$,
\item $D$ and $D'$ are exhaustive,
\item $S$ and $S'$ are compact-bumpy,
\item $G$ and $G'$ are ample
\item $C^\mathsf{R}_Z\subseteq S$ and $C'^{\mathsf{R}}_{Z'}\subseteq S'$, and
\item $A$ and $A'$ are diagonally isomorphic graded semigroups
\end{enumerate}
then $G$ and $G'$ are isomorphic \'etale groupoids.
\end{cor}

\begin{proof}
If $\phi:A\rightarrow A'$ is a diagonal-preserving graded semigroup isomorphism, so its restriction to $R=\mathsf{N}(\mathsf{Z}(D)^\mathsf{R})^\mathsf{R}_{\mathsf{Z}(D)}$ (see \eqref{R}), which is thus diagonally graded isomorphic to $R'=\mathsf{N}(\mathsf{Z}(D')^\mathsf{R})^\mathsf{R}_{\mathsf{Z}(D')}$.  As $G$ is ample and $S$ is compact-bumpy, so is $R$.  Likewise, $R'$ is compact-bumpy so $G$ and $G'$ are isomorphic graded \'etale groupoids, by \autoref{GradedGroupoidIsomorphism}.
\end{proof}

Note \autoref{GradedGroupoidIsomorphism3} generalises \cite[Theorem 5.6]{Steinberg2019} in that our gradings can take values in a groupoid rather than just a group (as well as in the other respects mentioned after \autoref{GroupoidIsomorphism3}).

\newpage

\bibliography{maths}{}
\bibliographystyle{alphaurl}

\end{document}